\documentclass[12pt]{amsart}

\usepackage{graphics}
\usepackage{color}
\usepackage[a4paper, margin=3cm]{geometry}

\usepackage{amssymb,enumerate}
\usepackage{amsmath}
\usepackage{bbm}           
\usepackage{bm}
\usepackage{eso-pic,graphicx}
\usepackage{tikz}
\usepackage{cite}
\usepackage{esint}
\usepackage[colorlinks=true, pdfstartview=FitV, linkcolor=blue, citecolor=blue, urlcolor=blue]{hyperref}
\usepackage{booktabs}
\usepackage{graphicx}
\usepackage{amsmath,amsfonts}
\usepackage{rotating}
\usepackage{amssymb}
\usepackage{verbatim}
\usepackage{rotating}
\usepackage{mathrsfs}

\makeatletter
\def\Ddots{\mathinner{\mkern1mu\raise\p@
		\vbox{\kern7\p@\hbox{.}}\mkern2mu
		\raise4\p@\hbox{.}\mkern2mu\raise7\p@\hbox{.}\mkern1mu}}
\makeatother

\def\XXint#1#2#3{{\setbox0=\hbox{$#1{#2#3}{\int}$}
		\vcenter{\hbox{$#2#3$}}\kern-.5\wd0}}

\begin{document}
	\newtheorem{theorem}{Theorem}
	\newtheorem{proposition}[theorem]{Proposition}
	\newtheorem{conjecture}[theorem]{Conjecture}
	\def\theconjecture{\unskip}
	\newtheorem{corollary}[theorem]{Corollary}
	\newtheorem{lemma}[theorem]{Lemma}
	\newtheorem{claim}[theorem]{Claim}
	\newtheorem{sublemma}[theorem]{Sublemma}
	\newtheorem{observation}[theorem]{Observation}
	\theoremstyle{definition}
	\newtheorem{definition}{Definition}
	\newtheorem{notation}[definition]{Notation}
	\newtheorem{remark}[definition]{Remark}
	\newtheorem{question}[definition]{Question}
	\newtheorem{questions}[definition]{Questions}
	\newtheorem{example}[definition]{Example}
	\newtheorem{problem}[definition]{Problem}
	\newtheorem{exercise}[definition]{Exercise}
	\newtheorem{thm}{Theorem}
	\newtheorem{cor}[thm]{Corollary}
	\newtheorem{lem}{Lemma}[section]
	\newtheorem{prop}[thm]{Proposition}
	\theoremstyle{definition}
	\newtheorem{dfn}[thm]{Definition}
	\theoremstyle{remark}
	\newtheorem{rem}{Remark}
	\newtheorem{ex}{Example}
	\numberwithin{equation}{section}
	\def\C{\mathbb{C}}
	\def\R{\mathbb{R}}
	\def\M{\mathbb{M}}
	\def\N{\mathbb{N}}
	\def\Q{{\mathbb{Q}}}
	\def\Z{\mathbb{Z}}
	\def\F{\mathcal{F}}
	\def\L{\mathcal{L}}
	\def\S{\mathcal{S}}
	\def\supp{\operatorname{supp}}
	\def\essi{\operatornamewithlimits{ess\,inf}}
	\def\esss{\operatornamewithlimits{ess\,sup}}
	
	\numberwithin{equation}{section}
	\numberwithin{thm}{section}
	\numberwithin{theorem}{section}
	\numberwithin{definition}{section}
	\numberwithin{equation}{section}
	
	\def\earrow{{\mathbf e}}
	\def\rarrow{{\mathbf r}}
	\def\uarrow{{\mathbf u}}
	\def\varrow{{\mathbf V}}
	\def\tpar{T_{\rm par}}
	\def\apar{A_{\rm par}}
	
	\def\reals{{\mathbb R}}
	\def\torus{{\mathbb T}}
	\def\scriptm{{\mathcal T}}
	\def\heis{{\mathbb H}}
	\def\integers{{\mathbb Z}}
	\def\z{{\mathbb Z}}
	\def\naturals{{\mathbb N}}
	\def\complex{{\mathbb C}\/}
	\def\distance{\operatorname{distance}\,}
	\def\support{\operatorname{support}\,}
	\def\dist{\operatorname{dist}\,}
	\def\Span{\operatorname{span}\,}
	\def\degree{\operatorname{degree}\,}
	\def\kernel{\operatorname{kernel}\,}
	\def\dim{\operatorname{dim}\,}
	\def\codim{\operatorname{codim}}
	\def\trace{\operatorname{trace\,}}
	\def\Span{\operatorname{span}\,}
	\def\dimension{\operatorname{dimension}\,}
	\def\codimension{\operatorname{codimension}\,}
	\def\nullspace{\scriptk}
	\def\kernel{\operatorname{Ker}}
	\def\ZZ{ {\mathbb Z} }
	\def\p{\partial}
	\def\rp{{ ^{-1} }}
	\def\Re{\operatorname{Re\,} }
	\def\Im{\operatorname{Im\,} }
	\def\ov{\overline}
	\def\eps{\varepsilon}
	\def\lt{L^2}
	\def\diver{\operatorname{div}}
	\def\curl{\operatorname{curl}}
	\def\etta{\eta}
	\newcommand{\norm}[1]{ \|  #1 \|}
	\def\expect{\mathbb E}
	\def\bull{$\bullet$\ }
	
	\def\blue{\color{blue}}
	\def\red{\color{red}}
	
	\def\xone{x_1}
	\def\xtwo{x_2}
	\def\xq{x_2+x_1^2}
	\newcommand{\abr}[1]{ \langle  #1 \rangle}

	\newcommand{\Norm}[1]{ \left\|  #1 \right\| }
	\newcommand{\set}[1]{ \left\{ #1 \right\} }
	\newcommand{\ifou}{\raisebox{-1ex}{$\check{}$}}
	\def\one{\mathbf 1}
	\def\whole{\mathbf V}
	\newcommand{\modulo}[2]{[#1]_{#2}}
	\def \essinf{\mathop{\rm essinf}}
	\def\scriptf{{\mathcal F}}
	\def\scriptg{{\mathcal G}}
	\def\scriptm{{\mathcal M}}
	\def\scriptb{{\mathcal B}}
	\def\scriptc{{\mathcal C}}
	\def\scriptt{{\mathcal T}}
	\def\scripti{{\mathcal I}}
	\def\scripte{{\mathcal E}}
	\def\scriptv{{\mathcal V}}
	\def\scriptw{{\mathcal W}}
	\def\scriptu{{\mathcal U}}
	\def\scriptS{{\mathcal S}}
	\def\scripta{{\mathcal A}}
	\def\scriptr{{\mathcal R}}
	\def\scripto{{\mathcal O}}
	\def\scripth{{\mathcal H}}
	\def\scriptd{{\mathcal D}}
	\def\scriptl{{\mathcal L}}
	\def\scriptn{{\mathcal N}}
	\def\scriptp{{\mathcal P}}
	\def\scriptk{{\mathcal K}}
	\def\frakv{{\mathfrak V}}
	\def\C{\mathbb{C}}
	\def\D{\mathcal{D}}
	\def\R{\mathbb{R}}
	\def\Rn{{\mathbb{R}^{n}}}
	\def\rn{{\mathbb{R}^{n}}}
	\def\Rm{{\mathbb{R}^{2n}}}
	\def\r2n{{\mathbb{R}^{2n}}}
	\def\Sn{{{S}^{n-1}}}
	\def\M{\mathbb{M}}
	\def\N{\mathbb{N}}
	\def\Q{{\mathcal{Q}}}
	\def\Z{\mathbb{Z}}
	\def\F{\mathcal{F}}
	\def\L{\mathcal{L}}
	\def\G{\mathscr{G}}
	\def\ch{\operatorname{ch}}
	\def\supp{\operatorname{supp}}
	\def\dist{\operatorname{dist}}
	\def\essi{\operatornamewithlimits{ess\,inf}}
	\def\esss{\operatornamewithlimits{ess\,sup}}
	\def\dis{\displaystyle}
	\def\dsum{\displaystyle\sum}
	\def\dint{\displaystyle\int}
	\def\dfrac{\displaystyle\frac}
	\def\dsup{\displaystyle\sup}
	\def\dlim{\displaystyle\lim}
	\def\bom{\Omega}
	\def\om{\omega}
	\author[T. Chen]{Ting Chen}
	\address{Ting Chen:
		School of Mathematical Sciences and LPMC\\
		Nankai University\\
		Tianjin, 300071\\
		People's Republic of China}
	\email{t.chen@nankai.edu.cn}

	\author[F. Liu]{Feng Liu$^{*}$}
	\address{Feng Liu:
		College of Mathematics and System Science\\
		Shandong University of Science and Technology\\
		Qingdao, Shandong 266590\\
		People's Republic of China}
	\email{FLiu@sdust.edu.cn}
	
	\author[Z. Wang]{Zhou Wang}
	\address{Zhou Wang:
		School of Mathematical Sciences and LPMC\\
		Nankai University\\
		Tianjin, 300071\\
		People's Republic of China}
	\email{1120240036@mail.nankai.edu.cn}
	
	\keywords{Weighted weak type $(1,\,1)$ bound, maximal singular integral,
		variation operator, jump operator.\\
		\indent{2020 Mathematics Subject Classification.} Primary 42B20, 42B25.}
	
	\thanks{$^{*}$ Corresponding author, e-mail address: FLiu@sdust.edu.cn}

	\date{\today}
	\title[Weighted Variational and Jump Inequalites]
	{Weighted Endpoint Variational and Jump Inequalites for Rough Singular Integrals}
	\maketitle
	
	\begin{abstract}
		In this present paper, we study the weighted weak type $(1,\,1)$ bounds
		for rough maximal singular integral operators as well as the the
		corresponding variation and jump operators. Our first result is the weighted weak type
		$(1,\,1)$ bound for the maximal singular integral
		$$T_{\Omega}^{*}f(x)=\sup\limits_{\varepsilon>0}|T_{\varepsilon,\Omega}f(x)|=\sup\limits_{\varepsilon>0}\bigg|\int_{|x-y|>\varepsilon}
		\frac{\Omega(x-y)}{|x-y|^{d}}f(y)dy\bigg|,$$
		where $\Omega\in L^\infty(\mathbb{S}^{d-1})$, homogeneous
		of degree zero, and satisfies the cancellation condition. We show that
		$$\|T_{\Omega}^{*}\|_{L^1(w)\rightarrow L^{1,\infty}(w)}\lesssim[w]_{A_1}[w]_{A_\infty}\log([w]_{A_\infty}+1),$$
		where $w$ belongs to the Muckenhoupt class $A_1(\mathbb{R}^d)$. This
		represents an essential improvement of a result (Honz\'{\i}k, Inter.
		Math. Res. Not. 2020) and a result (Bhojak and Mohanty, J. Funct. Anal.
		2023). Our second one is the weighted weak type $(1,\,1)$ bounds for
		variation and jump operators corresponding to $\{T_{\varepsilon,
			\Omega}\}_{\varepsilon\in2^{\mathbb{Z}}}$ and $\{T_{\varepsilon,
			\Omega}^{\phi}\}_{\varepsilon\in\mathbb{R}^{+}}$ under the condition
		$\Omega\in L^\infty(\mathbb{S}^{d-1})$, where $T_{\epsilon,\Omega}^{\phi}$
		represents a smooth truncation of rough singular integral operator.
		These results of this part are the {\it first} weighted weak type $(1,\,1)$
		variation inequalities and jump inequalities for rough singular integrals.
	\end{abstract}
	
	\bigskip
	
	\section{Introduction}\label{S1}
	
	\medskip
	
	The theory of singular integral operators is one of the core components in
	harmonic analysis. It has been successfully used in study partial differential
	equations, complex analysis and other fields. This topic was initiated in the
	seminal work of Calder\'{o}n and Zygmund \cite{CZ1}. The study of boundedness
	of rough singular integrals of convolution type has been an active area of
	research since the middle of the twentieth century. Let $\Omega$ be a homogenerous function of degree zero, Calder\'{o}n and Zygmund
	\cite{CZ2} first studied the rough singular integral
	$$T_{\Omega}f(x)=\lim\limits_{\varepsilon\rightarrow 0^{+}}T_{\varepsilon,\Omega}f(x)=\lim\limits_{\varepsilon\rightarrow 0^{+}}\int_{|x-y|>\varepsilon}\frac{\Omega(x-y)}{|x-y|^{d}}f(y)dy,$$
	and concluded that $T_\Omega$ is bounded on $L^p(\mathbb{R}^{d})$ for $1<p<\infty$,
	provided that $\Omega\in L\log^+ L(\mathbb{S}^{d-1})$ and satisfies the cancellation condition
	$$\int_{\mathbb{S}^{d-1}}\Omega(\theta)d\sigma(\theta)=0.\eqno(1.1)$$
	The above conclusion was later improved independently by Ricci and Weiss
	\cite{RW} and Connett \cite{Con} to the case $\Omega\in H^1(\mathbb{S}^{d-1})$
	(the Hardy space). At the endpoint $p=1$, the $L^1$ behavior of $T_\Omega$
	becomes a challenging issue. The weak type $(1,\,1)$ bound for $T_\Omega$
	was proved independently by Christ \cite{Ch} and Hofmann \cite{Ho1} under
	the condition that $\Omega\in L^q(\mathbb{S}^{d-1})$ for some $q>1$ and
	$d=2$. Meanwhile, Christ and Rubio de Francia \cite{CR} used $TT^{*}$ methods
	to improve the above results to the case $\Omega\in L\log^+ L(\mathbb{S}^{1})$.
	The above authors also mentioned that the above endpoint bound also holds for
	$d\leq 5$ in an unpublished paper. It is exciting that the weak type $(1,\,1)$
	bound for $T_\Omega$ with $\Omega\in L\log^+L(\mathbb{S}^{d-1})$ was
	affirmatively answered by Seeger \cite{See1} via a microlocal decomposition
	method in all dimensions $d\geq 2$ (also see \cite{DL,Tao}). A Stein's
	conjecture (see \cite{St}) focusing on the weak type $(1,\,1)$ bound for
	$T_\Omega$ with $\Omega\in H^1(\mathbb{S}^{d-1})$ is still an open problem.
	Another important direction is to investigate the $L^p$ theory of singular
	integral operators in weighted setting. We can consult \cite{Duo,Wat,CCDO,HRT,LPRR}
	for weighted $L^p$ estimates of $T_\Omega$ with
	$\Omega\in L^\infty(\mathbb{S}^{d-1})$. Particularly, Li et al. \cite{LPRR}
	proved that
	$$\|T_\Omega\|_{L^1(w)\rightarrow L^{1,\infty}(w)}\lesssim[w]_{A_1}[w]_{A_\infty}\log([w]_{A_\infty}+1),\ \ \ w\in A_1(\mathbb{R}^{d}).\eqno(1.2)$$
	
	It is well known that the following relations are valid.
	$$L^\infty(\mathbb{S}^{d-1})\subsetneq L^r(\mathbb{S}^{d-1})\,(1<r<\infty)\subsetneq L\log^+L(\mathbb{S}^{d-1})\subsetneq H^1(\mathbb{S}^{d-1})\subsetneq L^1(\mathbb{S}^{d-1}).$$
	
	One of the primary aims of this paper is to establish some new weighted
	endpoint bound for the rough maximal singular integral
	$$T_{\Omega}^{*}f(x)=\sup\limits_{\varepsilon>0}|T_{\varepsilon,\Omega}f(x)|,\ \ x\in\mathbb{R}^{d}.$$
	Calder\'{o}n and Zygmund \cite{CZ2} used the rotation method to establish
	the $L^p(\mathbb{R}^{d})\,(1<p<\infty)$ boundness for $T_{\Omega}^{*}$ when
	$\Omega\in L\log^{+}L(\mathbb{S}^{d-1})$. Compared with $T_\Omega$, the
	endpoint boundedness for $T_\Omega^{*}$ is more difficult. Honz\'{\i}k
	\cite{Hon} firstly concluded that $T_{\Omega}^{*}$ is bounded from
	$L(\log\log L)^{2+\epsilon}(\mathbb{R}^{d})$ to $L^{1,\infty}(\mathbb{R}^{d})$
	locally when $\Omega\in L^\infty(\mathbb{S}^{d-1})$ for any $\epsilon>0$. Here
	$L(\log\log L)^\alpha(\mathbb{R}^{d})\,(\alpha>0)$ is the set of all
	measurable functions $f:\mathbb{R}^{d}\rightarrow\mathbb{R}$ satisfying
	$$\|f\|_{L(\log\log L)^\alpha(\mathbb{R}^{d})}:=
	\inf\Big\{\lambda>0:\int_{\mathbb{R}^{d}}|f(x)|\Big(\log\log\Big({\rm e}^2+\frac{|f(x)|}{\lambda}\Big)\Big)^\alpha dx\leq\lambda\Big\}<\infty.$$
	When $\alpha=1$, we denote $L(\log\log L)^\alpha(\mathbb{R}^{d})=
	L\log\log L(\mathbb{R}^{d})$. Later on, Bhojak and Mohanty \cite{BM} improved
	the result of \cite{Hon} by establishing that $T_\Omega^{*}$ is bounded
	from $L\log\log L(\mathbb{R}^{d})$ to $L^{1,\infty}(\mathbb{R}^{d})$ locally under
	the condition $\Omega\in L\log^{+}L(\mathbb{S}^{d-1})$. Lai
	\cite{Lai} proved the weak type $(1,\,1)$ bound for $T_\Omega^{*}$ under
	the above rough kernel.  The $(1,p)$-sparse bound for $T^{*}_{\Omega}$ under
	the condition $\Omega\in L^{\infty}(\mathbb{S}^{d-1})$ was established in the very recent preprint of Wu \cite{Wu}  and $1<p<\infty$, which improves the previous $(1+\varepsilon,1+\varepsilon)$ and $(L\log\log L,L^{p})$ sparse bound established in \cite{HLP} and \cite{TH}.
	
	 The weighted $L^p\,(1<p<\infty)$ bound for
	$T_{\Omega}^{*}$ can be found in \cite{DHL}. For the weighted endpoint
	bound of $T_{\Omega}^{*}$, Bhojak and Mohanty \cite{BM} proved that the
	following result.
	
	\medskip
	
	\quad\hspace{-20pt}{\bf Theorem A} {\rm (\cite{BM})}. {\it Let $\Omega\in
		L^\infty(\mathbb{S}^{d-1})$ satisfy $(1.1)$. Then for $w\in A_1(\mathbb{R}^{d})$,
		$$\|T_{\Omega}^{*}f\|_{L\log\log L(w)\rightarrow L^{1,\infty}(w)}\lesssim[w]_{A_1}[w]_{A_\infty}(\log[w]_{A_\infty}+1).$$}
	\medskip
	
	Since $L\log\log L(w)\subsetneq L^1(w)$, which is a proper inclusion. A
	question which arises naturally in light of Theorem A and (1.1) is the
	following:
	
	\begin{question}\label{que1.1}
		Let $w\in A_1(\mathbb{R}^{d})$. Is $T_\Omega^{*}: L^1(w)\rightarrow
		L^{1,\infty}(w)$ bounded under the same condition of Theorem A?
	\end{question}
	
	This is one of the primary motivations of this paper. We shall provide
	an affirmative answer to the above question.
	
	\begin{theorem}\label{thm1.1}
		Let $\Omega\in L^\infty(\mathbb{S}^{d-1})$ satisfy $(1.1)$. Then
		for $w\in A_1(\mathbb{R}^{d})$,
		$$\|T_{\Omega}^{*}f\|_{L^1(w)\rightarrow L^{1,\infty}(w)}\lesssim [w]_{A_1}[w]_{A_\infty}(\log[w]_{A_\infty}+1).\eqno(1.3)$$
	\end{theorem}
	
	\begin{remark}\label{rem1.1}
		Theorem \ref{thm1.1} not only extend essentially (1.2), but also represents an essential
		improvement of Theorem A.
	\end{remark}

	Our another motivation of this paper is to study the weighted weak type
	$(1,\,1)$ bound for variation operator and jump operator of
	$\{T_{\varepsilon,\Omega}\}_{\varepsilon>0}$, which are larger essentially
	than rough maximal singular integral operator $T_{\Omega}^{*}$. The
	investigation on the boundedness of variational operators and jump
	operators has been an active topic of current research in harmonic
	analysis and ergodic theory. This topic began with L\'{e}pingle \cite{Le}
	who established the variational inequality for general martingales (see
	also \cite{PX} for a simple proof). Later, Bourgain \cite{Bou} proved
	similar variation estimates for the ergodic averages of a dynamic system.
	Since then, Bourgain's work \cite{Bou} stimulated lots of excellent works
	related to this topic. For examples, see \cite{JKRW,JRW} for the ergodic
	averages, \cite{MTX1,MTX2} for the differential operators, \cite{CJRW1,GT}
	for the Hilbert transform and Riesz transform, \cite{CJRW2,CDHL,DHL,JSW}
	for the singular integrals with rough kernels and \cite{MST,MTZ} for the
	discrete singular integral operators.

	Before stating our next result, let us recall some definitions, which
	followed from \cite{MSZ}. Let $\mathbb{I}$ be a totally ordered set.
	Given $\varrho\geq1$ and a family of complex numbers
	$\mathfrak{a}:=\{a_t\}_{t\in\mathbb{I}}$, the $\varrho$-variation norm
	of the family $\mathfrak{a}$ is defined by
	$$\|\mathfrak{a}\|_{V_\varrho}=\sup\limits_{t_1<t_2<\ldots<t_N,\,t_i\in\mathbb{I}}\Big(\sum\limits_{i=1}^{N-1}|a_{t_{i+1}}-a_{t_i}|^{\varrho}\Big)^{1/\varrho},\ \ 1\leq\varrho<\infty,\eqno(1.4)$$
	where the sup is taken over all finite increasing sequences $\{t_1<\cdots<t_N\}$
	with $t_i\in\mathbb{I}$. When $\varrho=\infty$, we define
	$$\|\mathfrak{a}\|_{V_\infty}=\sup\limits_{t_1<t_2,\,t_i\in\mathbb{I}}|a_{t_1}-a_{t_2}|.$$
	For each $\lambda>0$, the $\lambda$-jump counting function $\mathcal{N}_{\lambda}$
	of $\mathfrak{a}:t\in\mathbb{I}\mapsto a_{t}$ is defined by
	$$\begin{array}{ll}
	&\mathcal{N}_{\lambda}(\mathfrak{a}):=\sup \big\{N\in\mathbb{N}:\ \mathrm{There\ exists\ an\ increasing\ sequence}\ \\ &\qquad\qquad t_{1}<t_{2}<\cdots<t_{N+1},\ \mathrm{s. t.}\ \min_{1\leq i\leq N} |a_{t_{i}}-a_{t_{i+1}}| \geq \lambda\big\}.
	\end{array}$$
	For a measure space $(X,\mathscr{A},\mu)$, we consider the family of
	jump quasi-seminorms $J_{\varrho}^{p,q}$ on functions $F :(x,t)\in X\times
	\mathbb{I}\mapsto F_{t}(x)$ defined by
	$$J_{\varrho}^{p,q}((x,t)\mapsto F_{t}(x)):=\sup_{\lambda>0}\big\|\lambda(\mathcal{N}_{\lambda}(t \mapsto F_{t}(\cdot)))^{{1}/{\varrho}}\big\|_{L^{p,q}(X)}$$
	for $0<p<\infty$, $0<q<\infty$ and $0<\varrho<\infty$, where $L^{p,q}(X)$
	denotes Lorentz space. For $w\in A_{\infty}(\mathbb{R}^{d})$, let
	$J^{p,q}_{\varrho}(w,\varepsilon\in \mathbb{I})$ denote the jump quasi-seminorm
	for measure space $(\mathbb{R}^{d},\mathscr{M},wdx)$ and increasing sequences
	in $\mathbb{I}$.
	
	Precisely, let $\mathbb{R}^{+}:=(0,\infty)$ and $\mathcal{F}:=\{F_t\}_{t\in
		\mathbb{R}^{+}}$ be a family of Lebesgue measurable functions defined on
	$\mathbb{R}^{d}$. In view of the definition (1.4), one may define the strong
	$\varrho$-variation function $V_{\varrho}(\mathcal{F})(x)$ of a family
	$\mathcal{F}$ of functions. For any fixed $x\in\mathbb{R}^{d}$, the value of
	the strong $\varrho$-variation function $V_{\varrho}(\mathcal{F})$ of the
	family $\mathcal{F}$ at $x$ is defined by
	$$V_{\varrho}(\mathcal{F})(x)=\|\{F_t(x)\}_{t\in\mathbb{R}^{+}}\|_{V_\varrho},\ \ \varrho\geq1.$$
	The corresponding short $\varrho$-variation $\mathcal{S}_\varrho$ of the
	family $\mathcal{F}$ at $x$ is defined by
	$$\mathcal{S}_\varrho(\mathcal{F})(x)=\Big(\sum\limits_{k\in\mathbb{Z}}|V_{\varrho,k}(\mathcal{F})(x)|^{\varrho}\Big)^{1/\varrho},\eqno(1.5)$$
	where
	$$V_{\varrho,k}(\mathcal{F})(x)=\Big(\sup_{\substack{\varepsilon_{1}<\cdots<\varepsilon_{N}\atop
			[\varepsilon_{\ell},\varepsilon_{\ell+1}]\subset[2^{k},2^{k+1}]}}\sum_{\ell=1}^{N-1}|F_{\varepsilon_{\ell}}-F_{\varepsilon_{\ell+1}}|^{\varrho}\Big)^{1/\varrho}.\eqno(1.6)$$
	Moreover, for each $\lambda>0$, the $\lambda$-jump $N_{\lambda}$ corresponding
	to the family of functions $\mathcal{F}$ is defined by
	$$\begin{array}{ll}
		&N_\lambda(\mathcal{F})(x)=\sup\{N\in\mathbb{N}:\mathrm{There\ exists\ an\ increasing\ sequence}\\
		&\qquad\quad s_1<t_1\leq s_2<t_{2}\leq\ldots\leq s_N<t_N\ \ \mathrm{s. t.}\ \min_{1\leq \ell\leq N}|F_{s_\ell}(x)-F_{t_\ell}(x)|>\lambda\},
	\end{array}$$
	The dyadic $\lambda$-jump corresponding to $\mathcal{F}$ is defined as follows
	$$\begin{array}{ll}
		&N_\lambda^{{\rm dyad}}(\mathcal{F})(x)=\sup\{N\in\mathbb{N}:\mathrm{There\ exists\ an\ sequence\ of\ integers} \\
		&\quad j_1<k_1\leq j_2<k_2\leq\ldots\leq j_N<k_N\ \ \mathrm{s. t.}\ \min_{1\leq \ell\leq N}|F_{2^{j_\ell}}(x)-F_{2^{k_\ell}}(x)|>\lambda\}.
	\end{array}$$
	It was pointed out in \cite[p.6712, p.6716]{JSW} that for each $\lambda>0$,
	we have
	$$\mathcal{N}_{(1+\epsilon)\lambda}(\mathcal{F})(x)\leq N_\lambda(\mathcal{F})(x)\leq 2 \mathcal{N}_{\lambda/2}(\mathcal{F})(x),\eqno(1.7)$$
	where $\epsilon>0$,
	$$\lambda(\mathcal{N}_\lambda(\mathcal{F})(x))^{1/\varrho}\leq V_{\varrho}(\mathcal{F})(x),\ \lambda(N_\lambda(\mathcal{F})(x))^{1/\varrho}\leq 2^{1+1/\varrho}V_{\varrho}(\mathcal{F})(x),\ \ \varrho\geq1\eqno(1.8)$$
	and
	$$\lambda(N_\lambda(\mathcal{F})(x))^{1/\varrho}\leq C(\mathcal{S}_\varrho(\mathcal{F})(x)+\lambda(N_{\lambda/3}^{{\rm dyad}}(\mathcal{F})(x))^{1/\varrho}),\ \ \ \varrho\geq1.\eqno(1.9)$$
	In view of (1.7), one can pass from one definition of $J_{\varrho}^{p,q}$
	involving $\mathcal{N}_{\lambda}$ to the other definition involving
	$N_{\lambda}$ without difficulty. More details related to the relationship
	between variation operators and jump operators can be found in \cite{JSW}
	and \cite{MSZ}.
	
	In this paper we are concerned with the variation operators and jump
	operators for rough singular integrals. For convenience, we denote
	$\mathcal{T}_\Omega f=\{T_{\varepsilon,\Omega}f\}_{\varepsilon\in
		\mathbb{R}^{+}}$. In \cite{CJRW1,CJRW2}, the authors proved that
	$\Omega\in L\log^{+}L(\mathbb{S}^{d-1})$ is a sufficient condition
	for the valid of the following inequalities:
	$$\|V_\varrho(\mathcal{T}_{\Omega}f)\|_{L^p(\mathbb{R}^{d})}\lesssim \|f\|_{L^p(\mathbb{R}^{d})},\ \ \varrho>2,$$
	$$\|\lambda(N_\lambda(\mathcal{T}_{\Omega}f))^{1/\varrho}\|_{L^p(\mathbb{R}^{d})}\lesssim \|f\|_{L^p(\mathbb{R}^{d})},\ \ \varrho>2.\eqno(1.10)$$
	The inequality (1.10) with $\varrho=2$ was later proved by Jones et al.
	\cite{JSW} under the strong condition $\Omega\in L^q(\mathbb{S}^{d-1})$
	for some $q>1$. Subsequently, Ding et al. \cite{DHL} proved (1.10) also
	holds for $\varrho=2$ and $\Omega\in L\log^{+}L(\mathbb{S}^{d-1})$. For
	the $p=1$, Campbell et al. \cite{CJRW1,CJRW2} proved that
	$$\sup\limits_{\alpha>0}\alpha|\{x\in\mathbb{R}^{d}:V_\varrho(\mathcal{T}_{\Omega}f)(x)>\alpha\}|\lesssim\|f\|_{L^1(\mathbb{R}^{d})},\ \ \varrho>2,\eqno(1.11)$$
	$$\sup\limits_{\alpha>0}\alpha|\{x\in\mathbb{R}^{d}:\lambda(N_\lambda(\mathcal{T}_{\Omega}f))^{1/\varrho}(x)>\alpha\}|\lesssim \|f\|_{L^1(\mathbb{R}^{d})},\ \ \varrho>2,\eqno(1.12)$$
	under the condition that $\Omega$ belongs to the Lipschitz space
	${\rm Lip}_\alpha(\mathbb{S}^{d-1})$ for $\alpha>0$. In \cite{JSW}, Jones
	et al. pointed out that (1.12) with $\varrho=2$ also holds. The above
	authors also posed whether (1.11) and (1.12) also hold without the
	additional regularity assumption on $\Omega$. This question was recently
	solved by Bhojak and Shrivastava \cite{BS} who showed that (1.11) and (1.12)
	hold under the condition $\Omega\in L\log^{+}L(\mathbb{S}^{d-1})$. On the
	other hand, the weighted estimates for variation operators and jump
	operators of rough singular integrals has also been studied by many authors
	(see \cite{CDHL,MTX2}). Particularly, Chen et al. \cite{CDHL} established
	the following result.
	
	\medskip
	
	\quad\hspace{-20pt}{\bf Theorem B} {\rm (\cite{CDHL})} {\it Let
		$\Omega\in L^{q}(\mathbb{S}^{d-1})$ for some $q>1$ satisfying $(1.1)$. Then
		the following $\lambda$-jump inequality holds
		$$\Big\|\sup\limits_{\lambda>0}\lambda\sqrt{N_{\lambda}(\mathcal{T}_{\Omega}f)}\Big\|_{L^p(w)}\lesssim_{p,w}\|f\|_{L^{p}(w)},$$
		if $w$ and $p$ satisfy one of the following conditions:
		\begin{enumerate}[{\rm (i)}]
			\item $q'\leq p<\infty$, $p\neq 1$ and $w\in A_{p/q'}(\mathbb{R}^{d})$;
			\item $1<p\leq q$, $p\neq \infty$ and $w^{-\frac{1}{p-1}}\in A_{p'/q'}(\mathbb{R}^{d})$.
		\end{enumerate}
		Whence, for $\varrho>2$, the following variation inequality holds
		$$\|V_\varrho(\mathcal{T}_{\Omega}f)\|_{L^p(w)}\lesssim_{p,\varrho,w}\|f\|_{L^{p}(w)},$$
		if $w$ and $p$ satisfy one of the conditions $(i)$ or $(ii)$.}
	
	When $p=1$, Ma, Torrea and Xu \cite{MTX2} proved the following result
	(see \cite[Corollary 4]{MTX2}).
	
	\quad\hspace{-20pt}{\bf Theorem C} {\rm (\cite{MTX2})} {\it Let $\varrho>2$
		and $\Omega\in {\rm Lip}_\alpha(\mathbb{S}^{d-1})$ for some $\alpha>0$
		satisfying $(1.1)$. Then for $w\in A_1(\mathbb{R}^{d})$,
		$$\big\|V_\varrho(\mathcal{T}_{\Omega}f)\big\|_{L^{1,\infty}(w)}\lesssim\|f\|_{L^1(w)}.\eqno(1.13)$$}
	
	To our best knowledge, up to now, no proof has been known for the weighted
	weak type $(1,\,1)$ bound of variation operator and jump operator of
	singular integrals without the additional regularity assumption on $\Omega$.
	
	Based on Theorems B and C, it is natural to ask the following question.
	
	\begin{question}\label{que1.2}
		Let $w\in A_1(\mathbb{R}^{d})$. Whether the inequality (1.13) holds without
		the additional regularity assumption on $\Omega$.
	\end{question}
	
	This is another one of main motivations of this paper. In fact, we
	can't obtain the above answer to the above question. However, if we
	consider a class of special variation operators in dyadic setting,
	we shall provide an affirmative answer.
	
	\begin{theorem}\label{thm1.2}
		Let $\Omega\in L^\infty(\mathbb{S}^{d-1})$ satisfy $(1.1)$. Let
		$\mathcal{T}_{\Omega}^{{\rm dyad}}f=\{T_{\varepsilon,
			\Omega}f\}_{\varepsilon\in2^{\mathbb{Z}}}$. Then for $w\in A_1(\mathbb{R}^{d})$,
		$$\sup_{\lambda>0}\Big\|\lambda\sqrt{\mathcal{N}_\lambda(\mathcal{T}_{\Omega}^{{\rm dyad}}f)}\Big\|_{L^{1,\infty}(w)}\lesssim_{d,w}\|f\|_{L^1(w)}.\eqno(1.14)$$
		As a consequence, we have
		$$\big\|V_\varrho(\mathcal{T}_{\Omega}^{{\rm dyad}}f)\big\|_{L^{1,\infty}(w)}\lesssim_{d,w,\varrho}\|f\|_{L^1(w)}.\eqno(1.15)$$
	\end{theorem}
	\begin{remark}\label{rem1.2}
		From the montone convergence theorem and \cite[Lemma 2.3]{MSZ}, saying that the weak $L^{p}$ jump inequallity implies a weak $L^{p}$ estimate for variation seminorm for $p\in [1,\infty]$, 
		$$L^{1,\infty}(w;V^\varrho_{\varepsilon\in\mathbb{I}})\lesssim \varrho(\varrho-2)^{-1}J_{2}^{1,\infty}(w,\varepsilon\in\mathbb{I}),\eqno(1.16)$$
		where $\mathbb{I}$ is a totally ordered set. Then (1.15) follows from
		(1.14) and (1.16). On the other hand, by a standard limiting argument,
		(1.14) yields the boundedness of $T_\Omega^{*}:L^1(w)\rightarrow
		L^{1,\infty}(w)$ when $w\in A_1(\mathbb{R}^{d})$ and $\Omega\in L^\infty
		(\mathbb{S}^{d-1})$.
	\end{remark}
	
	It is worth mentioning that one can obtain a positive answer to Question
	\ref{que1.2} by consider a smooth truncation of rough singular integral
	operator. In 2013, Hyt\"{o}nen et al. \cite{HLP} firstly studied the
	weighted bounds for $\varrho$-variation operators of a smooth truncation
	of Calder\'{o}n--Zygmund singular integral operators. Precisely, let $\phi$
	be a smooth cut-off function satisfying
	$\one_{B_{1}^{c}}\leq\phi\leq\one_{B_{1/2}^{c}}$, where $B_{\varepsilon}
	=\{x\in\mathbb{R}^{d}:|x|\leq\varepsilon\}$. Clearly,
	$\supp(\phi)\subset[1/2,1]$. Let $K$ be a standard Calder\'{o}n--Zygmund
	kernel. We define the truncated operator
	$$T_{\varepsilon}^{\phi}f(x)=\int_{\mathbb{R}^{d}}K_{\varepsilon}^{\phi}(x,y)f(y)dy,\ \ \ \varepsilon>0,$$
	where
	$$K_{\varepsilon,\Omega}^{\phi}(x,y)=K(x,y)\phi\Big(\frac{x-y}{\varepsilon}\Big).\eqno(1.17)$$
	Let $\mathcal{T}_{\phi}f=\{T_{\varepsilon}^{\phi}f\}_{\varepsilon\in\mathbb{R}^{+}}$.
	Hyt\"{o}nen et al. \cite{HLP} studied the sharp weighted bounds for
	$V_\varrho(\mathcal{T}_{\phi})$ by assuming that $V_\varrho(\mathcal{T}_{\phi})$
	satisfies a prior weak type $(1,\,1)$ bound. Here we focus on the weighted
	weak type $(1,\,1)$ bound for $\varrho$-variation operator and
	$\lambda$-jump operator of a smooth truncation of rough singular integral
	operator
	$$T_{\varepsilon,\Omega}^{\phi}f(x)=\int_{\mathbb{R}^{d}}\frac{\Omega(x-y)}{|x-y|^{d}}\phi\Big(\frac{x-y}{\varepsilon}\Big)f(y)dy,\ \ \ \varepsilon>0.$$
	We have the following result.
	
	\begin{theorem}\label{thm1.3}
		Let $\Omega\in L^\infty(\mathbb{S}^{d-1})$ satisfy $(1.1)$. Let
		$\mathcal{T}_{\Omega}^{\phi}=\{T_{\varepsilon,\Omega}^{\phi}f\}_{
			\varepsilon\in\mathbb{R}^{+}}$. Then for $w\in A_1(\mathbb{R}^{d})$,
		$$\sup_{\lambda>0}\Big\|\lambda\sqrt{\mathcal{N}_\lambda(\mathcal{T}_{\Omega}^{\phi}f)}\Big\|_{L^{1,\infty}(w)}\lesssim_{d,w}\|f\|_{L^1(w)}.\eqno(1.18)$$
		As a consequence,
		$$\big\|V_\varrho(\mathcal{T}_{\Omega}^{\phi}f)\big\|_{L^{1,\infty}(w)}\lesssim_{d,w,\varrho}\|f\|_{L^1(w)}.\eqno(1.19)$$
	\end{theorem}
	
	\begin{remark}\label{rem1.3}
		Clearly, (1.19) follows from (1.16) and (1.18).
	\end{remark}
	
	\begin{remark}\label{rem1.4}
		Let $q>1$ and $w^{\frac{q}{q-1}}\in A_1(\mathbb{R}^{d})$. Repeating the same
		interpolation arguments between crucial Lebesgue measure estimate and trivial
		weighted estimate in the proofs of Theorems \ref{thm1.1}--\ref{thm1.3},
		combining the decomposition $\Omega=\Omega\one_{D_{s}}+\Omega\one_{D^{c}_{s}}$
		in \cite{See1}, we can deduce that Theorems \ref{thm1.1}--\ref{thm1.3}
		also hold under the condition that $\Omega\in L^q(\mathbb{S}^{d-1})$ and
		$w^{\frac{q}{q-1}}\in A_1(\mathbb{R}^{d})$. It should be noting that if
		$w^{\frac{q}{q-1}}\in A_1(\mathbb{R}^{d})$, then $w\in A_1(\mathbb{R}^{d})$
		since $M(w)^{\frac{q}{q-1}}(x)\leq M(w^{\frac{q}{q-1}})(x)\leq
		w^{\frac{q}{q-1}}(x)$ a.e. $x\in\mathbb{R}^{d}$.
	\end{remark}
	
	The main novelites of methods used to proving our main results are as
	follows.
	
	$\bullet$ Firstly, to get the better Lebesgue distributional estimate for
	long jump than the one in \cite[p. 9, line 3]{BS}, we obtain the strong
	$L^2$ estimate (2.13) for long jump via the recursion construction in \cite{Lai}. The  proof is simplified compared with that in \cite{Lai} by adopting the three shifted dyadic system in \cite{BS},  Our proof
	differs from the clever sublevel set arguments in \cite{BS},
	the latter obtains the desired Lebesgue distributional estimates by increasing the extra but harmless loss in the application of variational Rademacher--Menshov theorem and hence the collection of sets $\{F_{s,u}^{n}\}_{n\geq 1}$ arising in the section 2.4 
	is not constructed in \cite{BS}. However, compared with our $L^2$ estimate (2.13),
	the use of pigenhole principle in \cite{BS} leads to an additional
	loss in Lebesgue distributional estimates for long jump. In our proof, our
	use of pigenhole principle is employed in the interpolation step, rather
	than in the crucial Lebesgue distributional estimate step.
	
	$\bullet$ Secondly, to get the short jump estimate, we use the embedding
	arguments in \cite{JSW}, which are different from the arguments in
	\cite[Section 2.1]{BS}.

	This paper is organized as follows. In Section \ref{S2} we introduce
	some notation and establish some preliminary lemmas. In Section
	\ref{S3} we prove Theorem \ref{thm1.1}. The proof of Theorem \ref{thm1.2}
	will be given in Section \ref{S4}. Finally, we present the proof of
	Theorem \ref{thm1.3} in Section \ref{S5}.

	\medskip
	
	{\bf Notation:} In what follows, the notation $X\lesssim Y$ means
	$X\le C Y$ for some constant $C>0$ which is independent of the
	essential variables depending on $X$ and $Y$; and $X\simeq Y$ means
	$X\lesssim Y\lesssim X$. For each $E\subset\mathbb{R}^{d}$ we denote
	$E^c=\mathbb{R}^{d}\setminus E$. Denote by $\one_E$ the characteristic
	function on $E$. For any $p\in[1,\infty]$ we denote by $p'$ the dual
	exponent to $p$, i.e. $1/p+1/p'=1$. We set $p'=\infty$ when $p=1$
	and $p'=1$ when $p=\infty$. For $a\in\mathbb{R}$ we denote
	$[a]=\max\{N\in\mathbb{Z}:N\leq a\}$. Given $f\in L^p(\mathbb{R}^{d})$
	for $1\leq p\leq\infty$, we set $\|f\|_p=\|f\|_{L^p(\mathbb{R}^{d})}$.
	Let $\mathbb{N}=\{0,1,\ldots\}$. For $x=(x_1,\ldots,x_{d})\in\mathbb{R}^{d}$
	and $y=(y_1,\ldots,y_{d})\in\mathbb{R}^{d}$, we denote $\langle x,y\rangle
	=\sum_{i=1}^{d}x_iy_i$. For a cube $Q\subset\mathbb{R}^{d}$ we use the
	notation $\ell(Q)$ to denote the sidelength of $Q$.
	
	\bigskip
	
	\section{Preliminary}\label{S2}
	
	\subsection{Muckenhoupt weights}
	
	Let us start with $A_p(\mathbb{R}^{d})$ weights for $1\leq p\leq\infty$.

	\begin{definition}\label{def2.1} {\bf ($A_p$ weight)} (\cite{Gr}).
		A weight is a nonnegative, locally integrable function on $\mathbb{R}^{d}$
		that takes values in $(0,\infty)$ almost everywhere. For $1<p<\infty$, a
		weight $w$ is said to be in the Muckenhoupt weight class $A_p(\mathbb{R}^{d})$
		if
		$$[w]_{A_p}:=\sup\limits_{Q}\Big(\frac{1}{|Q|}\int_Qw(x)dx\Big)\Big(\frac{1}{|Q|}\int_Qw(x)^{1-p'}dx\Big)^{p-1}<\infty,$$
        where the supremum is over all cubes $Q\subset \mathbb{R}^{d}$. We call $[w]_{A_{p}}$ the $A_{p}$ constant. If $p=1$, 
		we say that $w\in A_{1}(\mathbb{R}^{d})$ if there exists a finite constant $C$ such that 
        $$ Mw(x)\leq Cw(x),\ a.e.,$$
        where $M$ is the Hardy--Littlewood maximal
		operator, then we define the $A_{1}$ constant as the infimum of all $C$ such that $ Mw(x)\leq Cw(x),\ a.e.$
		
		Now we define $A_{\infty}(\mathbb{R}^{d})=\bigcup_{p\geq 1}A_{p}(\mathbb{R}^{d})$,
		and the $A_{\infty}$ constant $[w]_{A_{\infty}}$ is defined by
		$$[w]_{A_{\infty}}:=\sup\limits_{Q}\frac{1}{w(Q)}\int_{Q}M(w\one_{Q})(x)dx,$$
		where $w(Q)=\int_Qw(x)dx$ and the supremum is over all cubes $Q\subset \mathbb{R}^{d}$.
	\end{definition}
	
	For $1\leq r<\infty$, let $M_r$ be a variant of the Hardy--Littlewood maximal
	operator, which is defined by
	$$M_rf(x)=\sup\limits_{Q}\Big(\frac{1}{|Q|}\int_{Q}|f(y)|^rdy\Big)^{1/r},$$
    where the supremum is over all cubes $Q\subset \mathbb{R}^{d}$.
	
	The following Lemma is the sharp reverse H\"{o}lder inequality for
	$A_{\infty}(\mathbb{R}^{d})$ weight.
	
	\begin{lemma}\label{lem2.1} {\rm{(\cite{HPR})}} Let
		$w\in A_{\infty}(\mathbb{R}^{d})$. There exists a dimensional constant
		$\tau_{d}$ such that
		$$\Big(\frac{1}{|Q|}\int_{Q}w^{r_{w}}(x)dx\Big)^{\frac{1}{r_{w}}}\leq\frac{2}{|Q|}\int_{Q}w(x)dx,$$
		where $r_{w}=1+\frac{1}{\tau_{d}[w]_{A_{\infty}}}$.
	\end{lemma}
	
	\subsection{Calder\'{o}n--Zygmund decomposition} Let $f\in L^1(w)$ and
	$\alpha>0$. By the Calder\'{o}n--Zygmund decomposition to $f$ at height
	$\alpha$, there exists a collection of standard dyadic cubes with disjoint
	interiors $\mathcal{Q}_{\alpha}=\{Q\in\mathscr{D}\}$, where
	$\mathscr{D}=\{2^{k}([0,1)^{d}+m):k\in\mathbb{Z},\,m\in\mathbb{Z}^d\}$
	and each $Q$ has sidelength $2^{L(Q)}$. There exist functions $g$ and
	$b$ such that
	\begin{enumerate}[{\rm (i)}]
		\item $f=g+b$;
		\item $\alpha<\int_Q |f|(x)dx\leq 2^{d}\alpha$;
		\item $w(\cup Q) \lesssim [w]_{A_{1}}\alpha^{-1}\|f\|_{L^1(w)}$ for all $w\in A_{1}$;
		\item $\|g\|_{\infty} \leq 2^{d}\alpha$;
		\item $\|g\|_{1,w} \lesssim [w]_{A_{1}}\|f\|_{L^1(w)}$;
		\item $b=\sum_{Q\in \mathcal{Q}_{\alpha}} b_{Q}$, $\supp(b_{Q})\subset Q$;
		\item Each $b_{Q}$ satisfies $\int_{Q}b_{Q}(x)dx=0$, $\|b_{Q}\|_{1}\lesssim 2^{d}\alpha|Q|$, $\|b_{Q}\|_{L^1(w)}\lesssim [w]_{A_{1}}\|f\|_{L^1(w)}$.
	\end{enumerate}
	Let $E=\cup_{Q\in\mathcal{Q}_{\alpha}}\tilde{Q}$, where $\tilde{Q}$
	denotes the cube that has the same center with $Q$ and satisfies
	$L(\tilde{Q})=L(Q)+300$. Clearly, $w(E)\lesssim_{d}[w]_{A_{1}}
	\alpha^{-1}\|f\|_{L^1(w)}$. Define $B_{j}$ as
	$$B_{j}:=\sum_{\substack{Q\in \mathcal{Q}^{\alpha}_{j}}}b_{Q}:=\sum_{\substack{Q\in \mathcal{Q}_{\alpha}\\ L(Q)=j}}b_{Q},\eqno(2.1)$$
	then $\mathcal{Q}^{\alpha}_{j-s}=\{Q\in \mathcal{Q}_{\alpha}:\ell(Q)=2^{j-s}\}$.
	
	\subsection{Microlocal decomposition of rough kernel}
	
	Let $\gamma\in(0,1)$ and $\Theta^s=\{e_\nu^s\}_{\nu\in\Lambda_s}$ be a
	collection of unit vectors on $\mathbb{S}^{d-1}$ such that
	\begin{enumerate}[{\rm (i)}]
		\item $|e_\nu^s-e_{\nu'}^s|>2^{-s\gamma-4}$ for any pair $(\nu,\nu')$
		with $\nu'\neq \nu$;
		\item For each $\theta\in\mathbb{S}^{d-1}$, there exists an $e_\nu^s$
		such that $|e_\nu^s-\theta|\leq 2^{-s\gamma-4}$.
	\end{enumerate}
	The set $\Theta^s$ can be constructed as in \cite{See1}. It was pointed
	out in \cite{See1} that $|\Theta^s|\lesssim 2^{s\gamma(d-1)}$. Let $\zeta$
	be a smooth, nonnegative, radial function, such that ${\rm supp}(\zeta)
	\subset B(0,1)$ and $\zeta(t)=1$ for $|t|\leq1/2$. Set
	$$\widetilde{\Gamma}_\nu^s(\xi)=\zeta(2^{s\gamma}({\xi}/{|\xi|}-e_\nu^s))$$
	and
	$$\Gamma_\nu^s(\xi)=\widetilde{\Gamma}_\nu^s(\xi)\Big(\sum_{\nu\in\Lambda_s}\widetilde{\Gamma}_\nu^s(\xi)\Big)^{-1}.\eqno(2.2)$$
	Clearly, $\Gamma_\nu^s$ is homogeneous of degree zero, and
	$\sum_{\nu\in\Lambda_s}\Gamma_\nu^s(\xi)=1$ for each $\xi\in\mathbb{R}^{d}
	\setminus\{0\}$ and $s$. Let $\varphi\in C_c^\infty(\mathbb{R}^{d})$ such
	that $0\leq\varphi\leq1$, $\supp(\varphi)\subset[-4,4]$ and
	$\varphi(t)\equiv1$ for $t\in[-2,2]$. We also define the following
	function, which is useful in our proof.
	$$\widehat{P_{\nu}^{s}}(\xi)=\varphi(2^{s\gamma}\langle\xi,e_\nu^s\rangle/|\xi|).\eqno(2.3)$$

	\subsection{A key estimate for variation operators}
	This subsection is devoted to presenting a precise estimate of variation
	operators. Let $\mathbb{I}$ be a totally ordered set and
	$\mathcal{F}:=\{F_t\}_{t\in \mathbb{I}}$ be a family of Lebesgue measurable
	functions defined on $\mathbb{R}^{d}$. We introduce the following variation
	operator
	$$\mathcal{V}_{2}(\mathcal{F})(x)=\sup\limits_{t\in  \mathbb{I}}|F_t(x)|+V_{2}(\mathcal{F})(x).\eqno(2.4)$$
	Let $\gamma$ be a radial function with $\gamma(x)=\phi(2|x|)-\phi(|x|)$,
	where $\phi$ is given in (1.17). Thus, $\gamma$ is supported on the
	annulus $\{x\in\mathbb{R}^{d}: 1/4\leq|x|\leq1\}$ and
	$\sum_{j\in\mathbb{Z}}\gamma_j(x)=1$ for any $x\neq0$, where
	$\gamma_j(x)=\gamma(2^{-j}x)$. Define the smooth truncated kernel $H_{j}$
	as
	$$H_{j}(x)=\frac{\Omega(x)}{|x|^{d}}\gamma_{j}(x)=K_{2^{j-1},\Omega}^{\phi}(x)-K_{2^{j},\Omega}^{\phi}(x),\eqno(2.5)$$
	where $K_{\varepsilon,\Omega}^{\phi}$ is defined in (1.17).
	
	Before presenting our main result of this subsection, let us introduce
	some notation and lemmas. Let $\mathscr{D}$ be the standard dyadic
	system, i.e. $\mathscr{D}=\{2^{j}[0,1)^{d}+m: j\in\mathbb{Z},m\in
	\mathbb{Z}^{d}\}$. Let $s\geq 200$, then for each $j\in\mathbb{Z}$ and $J\in \mathscr{D}$ with $L(J)=j$, define
	$B_{j-s,J}$ by
	$$B_{j-s,J}=\sum_{\substack{Q:Q\in \mathcal{Q}_{j-s}^{\alpha}\\Q\subset J }}b_{Q}.$$
	Therefore, we obtain
	$$H_{j}*B_{j-s}=\sum_{\substack{J:J\in \mathscr{D}\\ L(J)=j}}H_{j}*B_{j-s,J}:=\sum_{\substack{J:J\in \mathscr{D}\\ L(J)=j}}H_{j}*\sum_{\substack{Q:Q\in \mathcal{Q}_{j-s}^{\alpha}\\Q\subset J }}b_{Q}.$$
	Observe that $\supp(H_{j}*B_{j-s,J})\subset 4J$, where $4J$ denotes the
	cube that has the same center with cube $J$ and satisfies $\ell(4J)=4\ell(J)$.
	Now we introduce the shifted dyadic system, for each $u\in \{0,1/3,2/3\}^{d}$,
	the shifted dyadic system $\mathscr{D}^{u}$ is the collection of cubes
	$$\mathscr{D}^{u}:=\{2^{j}[0,1)^{d}+m+(-1)^{j}u: j\in\mathbb{Z},m\in\mathbb{Z}^{d}\}.$$
	
	Next we recall the following covering lemma in \cite{HLP}.
	
	\begin{lemma}{\rm(\cite[Lemma 2.5]{HLP})}\label{lem2.2}
		For any cube $J$ and $k\in\mathbb{N}$, there exists a shifted dyadic cube
		$R\in\mathscr{D}^{u}$ for some $u\in\{0,1/3,2/3\}^{d}$ such that $J\subset R$,
		$2^{k}J\subset R^{(k)}$, and $3\ell(J)<\ell(R)\leq6\ell(J)$, where $R^{(k)}$
		denotes the $k$ generation older dyadic ancestor of $R$ with respect to the
		system $\mathscr{D}^{u}$.
	\end{lemma}
	
	In view of Lemma \ref{lem2.2}, then for each cube $J$ with $\ell(J)=2^{j}$,
	there exist some $u\in\{0,1/3,2/3\}^{d}$ and a shifted cube
	$R\in\mathscr{D}^{u}$ such that $\ell(R)=2^{j+2}$, $J\subset R$,
	$2^{2}J\subset R^{(2)}\in\mathscr{D}^{u}$. Now let $u(J)=u$, $D(J)=R^{(2)}$,
	thus $D(J)\in\mathscr{D}^{u(J)}$ and $\ell(D(J))=2^{j+4}$. It is worth noting
	that there may exist $u'\neq u$ and $R'\in\mathscr{D}^{u'}$ satisfying
	$\ell(R')=2^{j+2}$, $J\subset R'$, $2^{2}J\subset R'^{(2)}$, but we only take
	one of such $(u,R\in\mathscr{D}^{u})$ as the value of $(u(J),D(J))$, hence
	$u(J)$ and $D(J)$ are well-defined. By the definition of $(u(J),D(J))$, we
	write
	$$H_{j}*B_{j-s}=\sum\limits_{\substack{J:J\in \mathscr{D}\\ L(J)=j}}H_{j}*B_{j-s,J}=\sum\limits_{u\in\{0,\frac{1}{3},\frac{2}{3}\}^{d}}
	\sum\limits_{\substack{J\in \mathscr{D}:u(J)=u\\ L(J)=j}}H_{j}*B_{j-s,J}.\eqno(2.6)$$
	
	In what follows, we set $u\in \{0,1/3,2/3\}^{d}$ and $s\geq 200$. We now
	define two important families of sets, which followed from \cite{Lai}.
	
	\begin{definition}\label{def2.2}
		\begin{enumerate}[{\rm (i)}]
			\item ($\{F_{s,u}^{n}\}_{n\geq 1}$). First of all, we set
			$$\mathcal{I}^{0}_{s,u}=\big\{J\in \mathscr{D}: u(J)=u, \exists\ j\in\mathbb{Z}, Q\in \mathcal{Q}^{\alpha}_{j-s}\ \mathrm{s.t.\ } L(J)=j, Q\cap J\neq \emptyset\big\}.$$
			then we define the set $F_{s,u}^{1}$ as
			$$F_{s,u}^{1}=\Big\{x\in\mathbb{R}^{d}: \sum\limits_{J\in \mathcal{I}^{0}_{s,u}}\one_{D(J)}>C_{0}2^{ds}\Big\},$$
			where $C_{0}$ is a constant to be chosen later. For $n\geq 2$, we define
			the set $F_{s,u}^{n}$ successively as follows:
			$$F_{s,u}^{n}=\Big\{x\in\mathbb{R}^{d}: \sum\limits_{\substack{J\in \mathcal{I}^{0}_{s,u}\\D(J)\subset F_{s,u}^{n-1}} }\one_{D(J)}>C_{0}2^{ds}\Big\}.$$
			By the definition of $F_{s,u}^{n}$, there exists a chain $F_{s,u}^{0}\supset
			F_{s,u}^{1}\supset \cdots\supset F_{s,u}^{n}\supset \cdots$.
			
			\item ($\{\mathcal{I}^{n}_{s,u}\}_{n\geq 1}$). Define the sets
			$\{\mathcal{I}^{n}_{s,u}\}_{n\geq 1}$ as follows:
			$$\mathcal{I}^{1}_{s,u}=\{J\in \mathcal{I}^{0}_{s,u}:\, D(J)\not\subset F_{s,u}^{1}\},$$
			$$\mathcal{I}^{n}_{s,u}=\{J\in \mathcal{I}^{0}_{s,u}:\, D(J)\not\subset F_{s,u}^{n},\, D(J)\subset F_{s,u}^{n-1}\},\ \ n\geq2.$$
		\end{enumerate}
	\end{definition}
	In view of the above definitions, we write
	$$H_{j}*B_{j-s}=\sum\limits_{u\in\{0,\frac{1}{3},\frac{2}{3}\}^{d}}\sum\limits_{\substack{J\in \mathscr{D}:u(J)=u\\ L(J)=j}}H_{j}*B_{j-s,J}
	=\sum\limits_{u\in\{0,\frac{1}{3},\frac{2}{3}\}^{d}}\sum\limits_{\substack{J\in \mathcal{I}_{s,u}^{0}\\ L(J)=j}}H_{j}*B_{j-s,J}.\eqno(2.7)$$
	
	We shall establish the following lemma, which is analogous to Lemma
	3.3 and Lemma 3.6 in \cite{Lai}.
	
	\begin{lemma}\label{lem2.3}
		\begin{enumerate}[{\rm (i)}]
			\item For any measurable set $F$, we have
			$$\sum\limits_{j\in\mathbb{Z}}\sum\limits_{\substack{J\in \mathcal{I}^{0}_{s,u},\,D(J)\subset F\\ L(J)=j}}|D(J)|\lesssim 2^{ds}|F|.\eqno(2.8)$$
			Moreover, the following estimate also holds
			$$\sum\limits_{j\in\mathbb{Z}}\sum\limits_{\substack{J\in \mathcal{I}^{0}_{s,u}\\L(J)=j}}|D(J)|\lesssim 2^{ds}\sum\limits_{j\in\mathbb{Z}}\sum\limits_{Q\in \mathcal{Q}^{\alpha}_{j-s}}|Q|.\eqno(2.9)$$
			\item For each $s\geq 200$ and $n\geq 1$, we have
			$$|F_{s,u}^{n}|\lesssim 2^{-2d}\sum\limits_{j\in\mathbb{Z}}\sum\limits_{Q\in \mathcal{Q}^{\alpha}_{j-s}}|Q|.\eqno(2.10)$$
		\end{enumerate}
	\end{lemma}
	\begin{proof}
		Let $J\in \mathcal{I}^{0}_{s,u}$ and $L(J)=j$. There exists some
		$Q\in\mathcal{Q}_{j-s}^{\alpha}$ such that $Q\subset J$. Therefore,
		we obtain
		$$\begin{array}{ll}
			&\displaystyle\sum\limits_{j\in\mathbb{Z}}\sum\limits_{\substack{J\in \mathcal{I}^{0}_{s,u},\,D(J)\subset F\\ L(J)=j}}|D(J)|
			=\displaystyle\sum\limits_{j\in\mathbb{Z}}\sum\limits_{\substack{J:D(J)\in \mathscr{D}^{u},\,D(J)\subset F\\(\cup_{Q\in \mathcal{Q}_{j-s}}Q)\cap J\neq \emptyset,\\ L(J)=j}}|D(J)|\\
			&\qquad\qquad\qquad\qquad\leq\displaystyle2^{4d} 2^{sd} \sum\limits_{j\in \mathbb{Z}}\sum\limits_{\substack{J:D(J)\in \mathscr{D}^{u}:\,D(J)\subset F\\(\cup_{Q\in \mathcal{Q}_{j-s}}Q) \cap J\neq \emptyset,\\ L(J)=j}}\sum_{\substack{Q\in \mathcal{Q}^{\alpha}_{j-s}\\Q\subset J}}|Q|\\
			&\qquad\qquad\qquad\qquad\lesssim\displaystyle 2^{sd}\sum\limits_{j\in\mathbb{Z}}\sum\limits_{\substack{Q\in \mathcal{Q}^{\alpha}_{j-s}\\Q\subset F}}|Q|\\
			&\qquad\qquad\qquad\qquad\lesssim\displaystyle 2^{sd}\min\Big\{|F|,\,\sum\limits_{j\in\mathbb{Z}}\sum_{Q\in \mathcal{Q}^{\alpha}_{j-s}}|Q|\Big\}.
		\end{array}$$
		This proves (2.8) and (2.9).
		
		On the other hand, we set $C_{0}=10\cdot 2^{4d}$. When $n=1$, we get
		from (2.9) that
		$$|F_{s,u}^{1}|\leq\frac{1}{C_{0}2^{ds}}\sum\limits_{j\in\mathbb{Z}}\sum\limits_{\substack{J\in \mathcal{I}^{0}_{s,u}\\L(J)=j}}|D(J)|
		\lesssim\sum\limits_{j\in\mathbb{Z}}\sum\limits_{Q\in \mathcal{Q}^{\alpha}_{j-s}}|Q|,$$
		which leads to (2.10) for $n=1$. When $n\geq 2$, by (2.8), we obtain the recursion estimate $|F_{s,u}^{n}|\leq 2^{-2}|F_{s,u}^{n-1}|$. This together with the
		estimate for $|F_{s,u}^{1}|$ implies (2.10) for $n\geq 2$.
	\end{proof}
	
	The following result is the variational Rademacher--Menshov theorem from
	\cite[Lemma 7.2]{DOP}.
	
	\begin{lemma}\label{lem2.4} {\rm (\cite[Lemma 7.2]{DOP})}. Let $\{F_i\}_{i=1}^N$
		be a sequence of functions on a measure space $X$ such that for every
		sequence of signs $\epsilon_1,\cdots, \epsilon_{N}$, there exists a
		constant $B>0$ such that
		$$\Big\|\sum\limits_{i=1}^N\epsilon_iF_i\Big\|_2\leq B,$$
		Then,
		$$\Big\|\mathcal{V}_2\Big(\Big\{\sum\limits_{i=1}^{n}F_i\Big\}_{1\leq n\leq N}\Big)\Big\|_2\lesssim(1+\log N) B.$$
	\end{lemma}
	Let $H_j$ be given in (2.5). We split ${H}_{j}$ in two parts by the following,
	$${H}_{j}=\sum_{\nu\in \Lambda_{s}}P_{\nu}^{s}*H^{s}_{j,\nu}+\Big({H}_{j}-\sum_{\nu\in \Lambda_{s}}P_{\nu}^{s}*{H}^{s}_{j,\nu}\Big),\eqno(2.11)$$
	where ${H}_{j,\nu}^s(x)={H}_{j}(x)\Gamma_{\nu}^{s}(x)$, $P_{\nu}^{s}$,
	$\Lambda_{s}$ and $\Gamma_{\nu}^{s}$ are defined in subsection 2.3. Set
	$$\mathfrak{M}_{N}=\sup\limits_{r>0}\sup\limits_{\theta\in\mathbb{S}^{d-1}}\sup\limits_{0\leq l\leq N}\sup\limits_{j\in\mathbb{Z}} r^{d+l}\big|\partial_{r}^{l}{H}_{j}(r\theta)\big|.\eqno(2.12)$$
	
	The following results are slightly modified version of \cite[Lemma 2.1-2.2]{See1}, which are stated and briefly proved in \cite{Lai}
	
	\begin{lemma}\label{lem2.5} Let $\mathscr{Q}$ be a collection of cubes
		$Q\in \mathscr{D}$ with disjoint interiors. Let
		$\mathscr{Q}_{j} =\{Q \in \mathscr{Q}: L(Q) = j\}$. Let
		$\mathfrak{b}_Q$ be integrable and supported on $Q$, satisfying
		$$\int_{Q}|\mathfrak{b}_{Q}(y)|dy\lesssim_d\alpha |Q|,$$
		then for each $\gamma\in(0,1)$, $s\geq 200$,
		$$\Big\|\sum\limits_{j\in\mathbb{Z}}\Big(\sum\limits_{\nu\in \Lambda_{s}}P_{\nu}^{s}*H^{s}_{j,\nu}\Big)*\Big(\sum\limits_{Q\in\mathscr{Q}_{j-s}}\mathfrak{b}_Q\Big)\Big\|^{2}_{2}\lesssim_d\mathfrak{M}_{0}^{2}2^{-s\gamma}\alpha\sum\limits_{j\in\mathbb{Z}}\sum\limits_{Q\in \mathscr{Q}_{j-s}}\|\mathfrak{b}_Q\|_{1}.$$
	\end{lemma}
	
	\begin{lemma}\label{lem2.6} Let $\mathscr{Q}$ be a collection of cubes
		$Q\in \mathscr{D}$ with disjoint interiors. Let
		$\mathscr{Q}_{j}=\{Q \in \mathscr{Q}: L(Q) = j\}$. Let $\mathfrak{b}_Q$
		be integrable and supported on $Q$; moreover, suppose that $\mathfrak{b}_Q$
		has zero average, then for each $\varepsilon_0,\,\gamma\in(0,1)$,
		$N_{1}\geq d+1$, $s\geq 200$,
		$$\Big\|\sum\limits_{\nu\in \Lambda_{s}}\sum\limits_{j\in \mathbb{Z}}\big(H^{s}_{j,\nu}-P_{\nu}^{s}*H^{s}_{j,\nu}\big)*\Big(\sum\limits_{Q\in\mathscr{Q}_{j-s}}\mathfrak{b}_Q\Big)\Big\|_{1}\lesssim_dC_{s,\gamma,\varepsilon_{0},N_{1}}\sum\limits_{j\in\mathbb{Z}}\sum\limits_{Q\in \mathscr{Q}_{j-s}}\|\mathfrak{b}_Q\|_{1},$$
		where $C_{s,\gamma}=\mathfrak{M}_{0}2^{-s\eta_{1}}+\mathfrak{M}_{N_{1}}
		2^{-s\eta_{2}}$, $\eta_{1}=1-\varepsilon_{0}$, $\eta_{2}=(\varepsilon_{0}
		-\gamma)N_{1}-d\varepsilon_{0}-d\gamma$.
	\end{lemma}
	It is worth noting that the cancellation of $\mathfrak{b}_Q$ is important 
	for the valid of Lemma \ref{lem2.6}. 
	
	We now end this section by establishing the following proposition.
	
	\begin{proposition}\label{pro2.1}
		Let $B_j$ be defined in $(2.1)$ and $H_j$ be defined in $(2.5)$.
		There exists a constant $\delta>0$ such that for each $s\geq 200$,
		$$\Big\|\mathcal{V}_{2}\Big(\Big\{\sum\limits_{j\geq k}H_{j}*B_{j-s}\Big\}_{k\in\mathbb{Z}}\Big)\Big\|_{2}
		\lesssim_{d}s2^{-2\delta s}\alpha\Big(\sum\limits_{j\in\mathbb{Z}}\sum\limits_{Q\in \mathcal{Q}^{\alpha}_{j-s}}|Q|\Big)^{1/2}.\eqno(2.13)$$
	\end{proposition}
	\begin{proof}
		To establish the $L^{2}$ norm estimate (2.13), we adopted the recursion
		arguments. In view of (2.10), we can conclude that
		$\mathcal{I}^{0}_{s,u}=\cup_{n=1}^{\infty}\mathcal{I}^{n}_{s,u}$. Define
		$$\mathcal{D}_{s,u}^{n}=\{D(J)\in \mathscr{D}^{u}:J\in \mathcal{I}_{s,u}^{n}\}.$$
		Let $\mathcal{M}^{n,i}_{s,u}$ be the maximal cubes in $\mathcal{D}_{s,u}^{n}$
		and $\mathcal{M}^{n,2}_{s,u}$ be the collection of maximal cubes in
		$\mathcal{D}_{s,u}^{n}\setminus \mathcal{M}^{n,1}_{s,u}$. Now we construct
		$\mathcal{M}^{n,i}_{s,u}$ inductively, define $\mathcal{M}^{n,i}_{s,u}$ by
		$$\mathcal{M}^{n,i}_{s,u}=\{\mathrm{the\ maximal\ cubes\ of\  }\mathcal{D}_{s,u}^{n}\setminus \cup_{\kappa=1}^{i-1}\mathcal{M}^{n,\kappa}_{s,u}\}.$$
		By the algorithm of $\mathcal{I}_{s,u}^{n}$ above, we have that the
		construction of $\mathcal{M}^{n,i}_{s,u}$ will end with $u\geq C_{0}2^{ds}+1$.
		Since the element in $\mathcal{M}^{n,i}_{s,u}$ are mutually disjoint, then
		for each $x\in\mathbb{R}^{d}$ and $1\leq i\leq C_{0}2^{ds}$, either there is
		no $J$ so that $D(J)\in\mathcal{M}^{n,i}_{s,u}$ containing $x$ or there
		exists exactly one $K\in\mathcal{M}^{n,i}_{s,u}$ containing $x$. In the
		latter case, there exist an integer $i(x)\in [1,C_{0}2^{ds}]$ and a chain
		of cubes $K_{i(x)}\supsetneq\cdots\supsetneq K_{1}\ni x$ such that
		$K_{\kappa}\in\mathcal{M}^{n,\kappa}_{s,u}$, $1\leq\kappa\leq i(x)$.
		Therefore, we obtain that for each for each $x$ (similar to the maximal
		case, see page 14 in \cite{Lai})
		$$\begin{array}{ll}
			&\displaystyle \sup\limits_{k\in\mathbb{Z}}\Big|\sum\limits_{j\geq k}\sum_{\substack{J\in \mathcal{I}^{n}_{s,u}\\L(J)=j}}H_{j}*B_{j-s,J}(x)\Big|\\
			&\qquad\qquad\qquad\qquad\leq\displaystyle\mathcal{V}_{2}\bigg(\Big\{\sum\limits_{i=1}^{l}\sum\limits_{\substack{J:J\in \mathcal{I}^{n}_{s,u},\\D(J)\in\mathcal{M}^{n,i}_{s,u}}}H_{L(J)}*B_{L(J)-s,J}(x)\Big\}_{1\leq l\leq C_02^{ds}}\bigg),\\
			&V_{2}\bigg(\Big\{\sum\limits_{j\geq k}\sum\limits_{\substack{J\in \mathcal{I}^{n}_{s,u}\\L(J)=j}}H_{j}*B_{j-s,J}(x)\Big\}_{k\in\mathbb{Z}}\bigg)\\
			&\qquad\qquad\qquad\qquad\leq\displaystyle\mathcal{V}_{2}\bigg(\Big\{\sum\limits_{i=1}^{l}\sum\limits_{\substack{J:J\in \mathcal{I}^{n}_{s,u},\\D(J)\in\mathcal{M}^{n,i}_{s,u}}}H_{L(J)}*B_{L(J)-s,J}(x)\Big\}_{1\leq l\leq C_02^{ds}}\bigg).
		\end{array}\eqno(2.14)$$
		In view of the definition of $\mathcal{M}_{n,i}^{s,u}$, we can write
		$$\begin{array}{ll}
			&\displaystyle\sum\limits_{i=1}^{C_{0}2^{ds}}\epsilon_{i}\sum\limits_{\substack{J:J\in \mathcal{I}^{n}_{s,u}\\D(J)\in \mathcal{M}^{n,i}_{s,u}}}H_{L(J)}*B_{L(J)-s, J}
			=\displaystyle\sum\limits_{i=1}^{C_{0}2^{ds}}\epsilon_{i}\sum_{j\in\mathbb{Z}}\sum\limits_{\substack{J:J\in \mathcal{I}^{n}_{s,u},\\D(J)\in \mathcal{M}^{n,i}_{s,u}\\L(J)=j}}H_{j}*B_{j-s, J}\\
			&=\displaystyle\sum\limits_{j\in\mathbb{Z}}\sum\limits_{\substack{J:J\in \mathcal{I}^{n}_{s,u},\\D(J)\in\mathcal{M}^{n,i}_{s,u}\\L(J)=j}}\epsilon_{i(J)}H_{j}*B_{j-s, J}=\displaystyle\sum\limits_{\substack{J:J\in \mathcal{I}^{n}_{s,u},\\D(J)\in\mathcal{M}^{n,i}_{s,u}}}\epsilon_{i(J)}H_{L(J)}*B_{L(J)-s, J},
		\end{array}$$
		where $\epsilon_{i}\in\{-1,1\}$ and $\epsilon_{i(J)}=\epsilon_{i}$ if $J\in \mathcal{I}^{n}_{s,u}$, 
		$D(J)\in\mathcal{M}^{n,i}_{s,u}$. By (2.14) and Lemma \ref{lem2.4}, we obtain
		$$\begin{array}{ll}
			&\displaystyle\Big\|\mathcal{V}_2\Big(\Big\{\sum\limits_{j\geq k}\sum\limits_{\substack{J\in \mathcal{I}^{n}_{s,u}\\L(J)=j}}H_{j}*B_{j-s,J}\Big\}_{k\in\mathbb{Z}}\Big)\Big\|_{2}\\
			&\qquad\leq\displaystyle\Big\|\mathcal{V}_2\Big(\Big\{\sum\limits_{i=1}^{l}\sum\limits_{\substack{J:J\in \mathcal{I}^{n}_{s,u},\\D(J)\in\mathcal{M}^{n,i}_{s,u}}}H_{L(J)}*B_{L(J)-s, J}\Big\}_{1\leq l\leq C_02^{ds}}\Big)\Big\|_{2}\\
			&\qquad\qquad\lesssim\displaystyle s\sup\limits_{\epsilon_{i}\in \{-1,1\}}\Big\|\sum\limits_{i=1}^{C_{0}2^{ds}}\epsilon_{i}\sum\limits_{\substack{J:J\in \mathcal{I}^{n}_{s,u},\\D(J)\in\mathcal{M}^{n,i}_{s,u}}}H_{L(J)}*B_{L(J)-s, J}\Big\|_{2}\\
			&\leq\displaystyle s\sup\limits_{\epsilon_{J}\in \{-1,1\}}\Big\|\sum\limits_{J\in \mathcal{I}^{n}_{s,u}}\epsilon_{J}H_{L(J)}*B_{L(J)-s, J}\Big\|_{2}.
		\end{array}$$
		By the definition of $B_{j-s, J}$, we can write
		$$\sum\limits_{\substack{J\in \mathcal{I}^{n}_{s,u}\\L(J)=j}}\epsilon_{J}B_{j-s, J}
		=\sum\limits_{\substack{J\in \mathcal{I}^{n}_{s,u}\\L(J)=j}}\epsilon_{J}\sum_{\substack{Q\in \mathcal{Q}_{j-s}^{\alpha}\\Q\subset J} }b_{Q}.$$
		Further we obtain
		$$\begin{array}{ll}
			&\displaystyle\Big\|\mathcal{V}_2\Big(\Big\{\sum\limits_{j\geq k}\sum\limits_{\substack{J\in \mathcal{I}^{n}_{s,u}\\L(J)=j}}H_{j}*B_{j-s,J}\Big\}_{k\in\mathbb{Z}}\Big)\Big\|_2\\
			&\lesssim\displaystyle s\sup\limits_{\epsilon_{J}\in \{-1,1\}}\Big\|\sum_{j\in\mathbb{Z}}H_{j}*\Big(\sum\limits_{\substack{J\in \mathcal{I}^{n}_{s,u}\\L(J)=j}}\epsilon_{J}
			\sum\limits_{\substack{Q\in \mathcal{Q}_{j-s}^{\alpha}\\Q\subset J} }b_{Q}\Big)\Big\|_{2}.
		\end{array}\eqno(2.15)$$
		Next we prove that for some $\delta>0$,
		$$\sup\limits_{\epsilon_{J}\in \{-1,1\}}\Big\|\sum\limits_{j\in\mathbb{Z}}H_{j}*\Big(\sum\limits_{\substack{J\in \mathcal{I}^{n}_{s,u}\\L(J)=j}}
		\epsilon_{J}\sum\limits_{\substack{Q\in \mathcal{Q}_{j-s}^{\alpha}\\Q\subset J} }b_{Q}\Big)\Big\|_{2}\lesssim_{d}2^{-n}2^{-2\delta s}\alpha\Big(\sum\limits_{j\in\mathbb{Z}}\sum\limits_{Q\in \mathcal{Q}^{\alpha}_{j-s}}|Q|\Big)^{1/2}.\eqno(2.16)$$
		Now for each $n\geq 1$, $s\geq 200$, we define
		$$\mathfrak{Q}_{j-s}^{\alpha,n,s}=\{Q\in \mathcal{Q}_{j-s}^{\alpha}:\exists \ J\in \mathcal{I}^{n}_{s,u}, L(J)=j \ \mathrm{s.t.}\ J\supset Q\}.$$
		Write
		$$\sum\limits_{\substack{J\in \mathcal{I}^{n}_{s,u}\\L(J)=j}}\epsilon_{J}\sum\limits_{\substack{Q\in \mathcal{Q}_{j-s}^{\alpha}\\Q\subset J} }b_{Q}
		=\sum\limits_{Q\in \mathcal{Q}_{j-s}^{\alpha}}\sum\limits_{\substack{J\in \mathcal{I}^{n}_{s,u},\\J\supset Q, L(J)=j}}\epsilon_{J}b_{Q}
		=\sum\limits_{Q\in  \mathfrak{Q}_{j-s}^{\alpha,n,s}}\epsilon_{J(Q)}b_{Q},\eqno(2.17)$$
		where $J(Q)$ denotes the cube $J\in \mathcal{I}^{n}_{s,u}$ with $J\supset Q$ and
		$L(J)=j$, which is determined uniquely by $Q\in\mathfrak{Q}_{j-s}^{\alpha,n,s}$,
		so the value $J(Q)$ is well-defined for each $Q\in \mathfrak{Q}_{j-s}^{\alpha,n,s}$.
		
		Let $H^{s}_{j,\nu}$ be given in (2.11) and $\mathfrak{M}_N$ be given in (2.12).
		Applying \cite[v3, Lemma 4.6]{Lai}, we get the trivial bound
		$$\Big\|\sum\limits_{\nu\in \Lambda_{s}}\sum\limits_{j\in\mathbb{Z}}\big(H^{s}_{j,\nu}-P_{\nu}^{s}*H^{s}_{j,\nu}\big)*\Big(\sum\limits_{Q\in\mathscr{Q}_{j-s}}\mathfrak{b}_Q\Big)\Big\|^{3}_{3}\lesssim_{d} C_{s,2} \alpha^{2}\sum\limits_{j\in\mathbb{Z}}\sum\limits_{Q\in \mathscr{Q}_{j-s}}\|\mathfrak{b}_Q\|_{1},\eqno(2.18)$$
		where $\mathscr{Q}$, $\mathscr{Q}_{j-s}$, and $\mathfrak{b}_Q$ satisfy the
		conditions in Lemma \ref{lem2.5} and $C_{s,2}=2^{2s\gamma(d-1)+3([\frac{d}{2}]+1)}$.
		Note that $\mathfrak{M}_{N}\lesssim_{N} 1$. Applying the interpolation between 
		two inequalities appearing in Lemmas \ref{lem2.5} and \ref{lem2.6} and (2.18), 
		one obtains that for each $s\geq 200$,
		$$\Big\|\sum\limits_{j\in\mathbb{Z}}H_{j}*\Big(\sum\limits_{Q\in\mathscr{Q}_{j-s}}\mathfrak{b}_Q\Big)\Big\|^{2}_{2}\lesssim_d
		(2^{-\gamma s}+2^{-{2\vartheta}_{1}s}+2^{-2\vartheta_{2}s})\alpha\sum\limits_{j\in\mathbb{Z}}\sum_{Q\in \mathscr{Q}_{j-s}}\|\mathfrak{b}_Q\|_{1},\eqno(2.19)$$
		where $\mathscr{Q}$, $\mathscr{Q}_{j-s}$, and $\mathfrak{b}_Q$ satisfy all 
		conditions in Lemmas \ref{lem2.5} and \ref{lem2.6} and
		$$\vartheta_{1}=\frac{1}{4}(1-\varepsilon_{0})-\frac{3}{4}\gamma\Big(\frac{2}{3}(d-1)+[\frac{d}{2}]+1\Big),$$
		$$\vartheta_{2}=\frac{1}{4}((\varepsilon_{0}-\gamma)N_{1}-(\gamma+\varepsilon_{0})n)-\frac{3}{4}\gamma\Big(\frac{2}{3}(d-1)+[\frac{d}{2}]+1\Big).$$
		For fixed $s\geq 200$, applying (2.16) and (2.19) with $\mathscr{Q}_{j-s}$ 
		replaced by $\mathfrak{Q}_{j-s}^{\alpha,n,s}$ and $\mathfrak{b}_{Q}$ replaced 
		by $\epsilon_{J(Q)}b_{Q}$, then for each $\gamma$, $\varepsilon_0\in(0,1)$ 
		and $N_{1}\geq d+1$,
		$$\Big\|\sum\limits_{j\in\mathbb{Z}}H_{j}*\Big(\sum\limits_{\substack{J\in \mathcal{I}^{n}_{s,u}\\L(J)=j}}\epsilon_{J}\sum\limits_{\substack{Q\in \mathcal{Q}_{j-s}^{\alpha}
				\\Q\subset J} }b_{Q}\Big)\Big\|^{2}_{2}\lesssim_{d}(2^{-\gamma s}+2^{-{2\vartheta}_{1}s}+2^{-2\vartheta_{2} s})\alpha^{2}\sum\limits_{j\in\mathbb{Z}}\sum\limits_{Q\in\mathfrak{Q}_{j-s}^{\alpha,n,s}}|Q|.$$
		
		Now we choose $0<\iota\ll\gamma\ll \varepsilon_{0}\ll1$ and $N_{1}\geq d+1$ so 
		that $\vartheta_1,\vartheta_2>0$. Note that $Q\in\mathcal{Q}_{j-s}^{\alpha}$ 
		are mutually disjoint and $J\subset D(J)\subset F_{s,u}^{n-1}$ for each 
		$J\in \mathcal{I}_{s,u}^{n}$. Then we have
		$$\begin{array}{ll}
			&\displaystyle\sum\limits_{j\in\mathbb{Z}}\sum\limits_{\substack{Q\in \mathfrak{Q}_{j-s}^{\alpha,n,s}}}|Q|
			=\displaystyle\sum\limits_{j\in\mathbb{Z}}\sum\limits_{\substack{J:J\in \mathcal{I}_{s,u}^{n}\\L(J)=j}}\sum\limits_{\substack{Q\in \mathfrak{Q}_{j-s}^{\alpha,n,s}\\Q\subset J}}|Q|\\
			&\qquad\lesssim\displaystyle\min\Big\{\Big|\bigcup_{j\in\mathbb{Z}}\bigcup_{\substack{J:J\in \mathcal{I}_{s,u}^{n}\\L(J)=j}}J\Big|,\sum\limits_{j\in\mathbb{Z}}\sum_{Q\in \mathcal{Q}^{\alpha}_{j-s}}|Q|\Big\}\\
			&\qquad\qquad\qquad\lesssim\displaystyle\min\Big\{|F_{s,u}^{n-1}|,\sum\limits_{j\in\mathbb{Z}}\sum\limits_{Q\in \mathcal{Q}^{\alpha}_{j-s}}|Q|\Big\}.
		\end{array}$$
		This together with Lemma \ref{lem2.3} yields (2.16).
		
		Now we get from (2.7), (2.15) and (2.16) that
		$$\begin{array}{ll}
			&\displaystyle\Big\|\mathcal{V}_2\Big(\Big\{\sum\limits_{j\geq k}H_{j}*B_{j-s}\Big\}_{k\in\mathbb{Z}}\Big)\Big\|_2
			=\displaystyle\Big\|\mathcal{V}_2\Big(\Big\{\sum_{j\geq k}\sum_{u\in\{0,\frac{1}{3},\frac{2}{3}\}^{d}}\sum_{\substack{J\in \mathcal{I}^{0}_{s,u}\\l(J)=2^{j}}}H_{j}*B_{j-s, J}\Big\}_{k\in\mathbb{Z}}\Big)\Big\|_2\\
			&\qquad\leq\displaystyle\sum\limits_{u\in\{0,\frac{1}{3},\frac{2}{3}\}^{d}}\Big\|\mathcal{V}_{2}\Big(\Big\{\sum\limits_{j\geq k}\sum_{\substack{J\in\mathcal{I}^{0}_{s,u}\\l(J)=2^{j}}}H_{j}*B_{j-s, J}\Big\}_{k\in\mathbb{Z}}\Big)\Big\|_2\\
			&\qquad\qquad=\displaystyle\sum\limits_{u\in\{0,\frac{1}{3},\frac{2}{3}\}^{d}}\Big\|\mathcal{V}_{2}\Big(\Big\{\sum\limits_{j\geq k}\sum\limits_{n=1}^{\infty}\sum\limits_{\substack{J\in \mathcal{I}^{n}_{s,u} \\l(J)=2^{j}}}H_{j}*B_{j-s,J}\Big\}_{k\in\mathbb{Z}}\Big)\Big\|_{2}\\
			&\qquad\qquad\qquad\leq\displaystyle\sum\limits_{u\in\{0,\frac{1}{3},\frac{2}{3}\}^{d}}\sum\limits_{n=1}^{\infty}\Big\|\mathcal{V}_{2}\Big(\Big\{\sum\limits_{j\geq k}\sum\limits_{\substack{J\in \mathcal{I}^{n}_{s,u} \\l(J)=2^{j}}}H_{j}*B_{j-s, J}\Big\}_{k\in\mathbb{Z}}\Big)\Big\|_{2}\\
			&\qquad\qquad\qquad\qquad\lesssim_{d}\displaystyle s2^{-2\delta s}\alpha\Big(\sum\limits_{j\in\mathbb{Z}}\sum\limits_{Q\in \mathcal{Q}^{\alpha}_{j-s}}|Q|\Big)^{1/2}.
		\end{array}$$
		This proves (2.11) and completes the proof of Proposition \ref{pro2.1}.
	\end{proof}
	
	\medskip
	
	\section{Proof of Theorem \ref{thm1.1}}\label{S3}
	
	In this section we present the proof of Theorem \ref{thm1.1}. Without 
	loss of generality we assume $\|\Omega\|_{L^\infty(\mathbb{S}^{d-1})}=1$. 
	Let $f\in L^1(w)$,  $w\in A_1$. By density, 
	we may assume $f\in C_c^\infty(\mathbb{R}^{d})$.

	Let $M_\Omega$ be the following maximal operator with rough $\Omega$
	$$M_\Omega f(x)=\sup\limits_{r>0}\frac{1}{r^{d}}\int_{|y|\leq r}|\Omega(y)||f(x-y)|dy.$$
	It is well known that $\|M_{\Omega}\|_{L^{1}(w)\rightarrow L^{1,\infty}(w)}
	\lesssim_{d}[w]_{A_{1}}\|\Omega\|_{L^\infty(\mathbb{S}^{d-1})}$. Thus, 
	for (1.3), it is sufficient to prove
	$$w\Big(\Big\{x\in\mathbb{R}^{d}:\sup_{k\in\mathbb{Z}}\Big|\sum_{j\geq k}H_{j}*g\Big|>\alpha\Big\}\Big)\lesssim_{d}C(w)\alpha^{-1}\|f\|_{L^1(w)}\eqno(3.1)$$
	$$w\Big(\Big\{x\notin E:\sup_{k\in\mathbb{Z}}\Big|\sum_{j\geq k}H_{j}*b\Big|>\alpha\Big\}\Big)\lesssim_{d}C(w)\alpha^{-1}\|f\|_{L^1(w)}, \eqno(3.2)$$
	for all $\alpha>0$, where the constant $C(w)$ denotes $[w]_{A_{1}}[w]_{A_{\infty}}\log_{2}([w]_{A_{\infty}}+1)$, 
	$H_j$ is defined in (2.5). Set
	$$r_{w}=1+{\frac{1}{\tau_{d}[w]_{A_{\infty}}}}<p_{w}=1+{\frac{1}{\log([w]_{A_{\infty}}+1)}}.$$
	In view of Fefferman--Stein type inequality (see \cite[Theorem 1.3]{DPHL}) 
	and Lemma \ref{lem2.1}, the inequality (3.1) follows from
	$$\begin{array}{ll}
		\displaystyle w\Big(\Big\{x\in\mathbb{R}^{d}:\sup_{k\in\mathbb{Z}}\Big|\sum_{j\geq k}H_{j}*g\Big|>\alpha\Big\}\Big)&\leq\displaystyle\alpha^{-p_{w}}\Big\|\sup_{k\in\mathbb{Z}}\Big|\sum_{j\geq k}H_{j}*g\Big|\Big\|^{p_{w}}_{L^{p_{w}}(w)}\\
		&\lesssim_{d}\displaystyle\alpha^{-p_{w}}p_{w}^{2p_{w}}p_{w}'(r'_{w})^{p_{w}(1+\frac{1}{p'_{w}})}\|g\|^{p_{w}}_{L^{p_{w}}(M_{r_{w}}w)}\\
		&\lesssim_{d}\alpha^{-p_{w}} p_{w}^{2p_{w}}p_{w}'(r'_{w})^{2p_{w}-1}\|g\|^{p_{w}}_{L^{p_{w}}(M_{r_{w}}w)}\\
		&\lesssim_{d}[w]_{A_{1}}[w]_{A_{\infty}}\log_{2}([w]_{A_{\infty}}+1)\alpha^{-1}\|f\|_{L^1(w)}.
	\end{array}$$
	On the other hand, observe that $H_{j}*B_{j-s}=0$ if $s<200$. Then (3.2) 
	follows from
	$$\begin{array}{ll}
		&\displaystyle w\Big(\Big\{x\in\mathbb{R}^{d}:\sum\limits_{s\geq 200}\sup_{k\in\mathbb{Z}}\Big|\sum_{j\geq k }H_{j}*B_{j-s}(x)\Big|>\alpha\Big\}\Big)\\
		&\lesssim_{d}[w]_{A_1}[w]_{A_\infty}(\log[w]_{A_\infty}+1){\alpha}^{-1}\|f\|_{L^1(w)}.
	\end{array}\eqno(3.3)$$
	
	Now we prove (3.3). First of all, for each 
	non-negative function $v$ and $s\geq 200$, by Fubini theorem and
	$\|b_{Q}\|_{1}\lesssim_{d}\alpha|Q|$, we have the following trivial 
	bound
	$$\begin{array}{ll}
		\displaystyle\Big\|\sup_{k\in\mathbb{Z}}\Big|\sum_{j\geq k}H_{j}*B_{j-s}\Big|\Big\|_{L^{1}(v)}&\lesssim\displaystyle\int_{\mathbb{R}^{d}}\sum_{j\in\mathbb{Z}}\int_{|x-y|\lesssim 2^{j}}|H_{j}(x-y)|v(x)dx|B_{j-s}(y)|dy\\
		&\lesssim_{d}\displaystyle\|\Omega\|_{L^\infty(\mathbb{S}^{d-1})}\sum\limits_{j\in\mathbb{Z}}\sum\limits_{\substack{Q\in \mathcal{Q}_{\alpha},\\L(Q)=j-s}}\int_{Q}|b_{Q}(y)|\inf_{z\in Q}Mv(z)dy\\
		&\lesssim_{d}\displaystyle\sum_{j\in\Z}\sum_{Q\in \mathcal{Q}_{j-s}^{\alpha}}\alpha|Q|\inf_{z\in Q}Mv(z).
	\end{array}$$
	This together with Chebyshev's inequality imply that
	$$v\Big(\Big\{x\in\mathbb{R}^{d}:\sup\limits_{k\in\mathbb{Z}}\Big|\sum_{j\geq k}H_{j}*B_{j-s}(x)\Big|>\epsilon\alpha\Big\}\Big)
	\lesssim_{d}\epsilon^{-1} \sum_{j\in\Z} \sum_{\substack{Q\in \mathcal{Q}_{j-s}^{\alpha}}}|Q|\inf_{z\in Q}Mv(z).$$
	Applying proposition \ref{pro2.1} and Chebyshev's inequality, we obtain  that for each $\epsilon\in(0,1)$ 
	and $\alpha>0$, there exists a constant $\delta>0$ such that for each 
	$s\geq200$, the following estimate holds
	$$\Big|\Big\{x\in\mathbb{R}^{d}:\sup_{k\in\mathbb{Z}}\Big|\sum\limits_{j\geq k}H_{j}*B_{j-s}(x)\Big|>\epsilon\alpha\Big\}\Big|\lesssim_d\epsilon^{-2}2^{-\delta s}\sum\limits_{j\in\mathbb{Z}}\sum\limits_{Q\in \mathcal{Q}^{\alpha}_{j-s}}|Q|.\eqno(3.4)$$
	 Now we shall apply the interpolation arguments 
	of Vargas \cite{Var} to get the desired weighted estimates. Following the 
	strategy of \cite{BM}, which is inspired by the method given in \cite{LPRR}, 
	we can prove that for each $u>0$ and non-negative function $v$, we have
	$$\int_{E_{\epsilon\alpha}^{s}}\min\{v(x),u\} dx\lesssim_{d}\epsilon^{-2}\sum_{Q\in \mathcal{Q}_{\alpha}}|Q|\min\Big\{u2^{-\delta s},\inf_{z\in Q}Mv(z)\Big\},\eqno(3.5)$$
	where
	$$E_{\epsilon\alpha}^{s}=\{x\in\mathbb{R}^{d}:\sup_{k\in\mathbb{Z}}|\sum\limits_{j\geq k}H_{j}*B_{j-s}(x)|>\epsilon\alpha\}.$$
	Indeed, for each $u>0$, we set
	$$\mathcal{Q}^{1,\alpha}_{j-s,u}=\Big\{Q\in \mathcal{Q}_{\alpha}: L(Q)=j-s,\inf_{z\in Q}Mv(z)\leq 2^{-\delta s}u\Big\},$$
	$$\mathcal{Q}^{2,\alpha}_{j-s,u}=\Big\{Q\in\mathcal{Q}_{\alpha}: L(Q)=j-s,\inf_{z\in Q}Mv(z)>2^{-\delta s}u\Big\}.$$
	For each $j$, split $\mathcal{Q}^{\alpha}_{j-s}=\mathcal{Q}^{1,\alpha}_{j-s,u}
	\cup\mathcal{Q}^{2,\alpha}_{j-s,u}$, then we can write $B_{j-s}=B^{1,
		\alpha}_{j-s,u}+B^{2,\alpha}_{j-s,u}$, where
	$$B^{1,\alpha}_{j-s,u}=\sum_{Q\in \mathcal{Q}_{j-s,u}^{1,\alpha}}b_{Q},\ \ \ \ \ B^{2,\alpha}_{j-s,u}=\sum_{Q\in \mathcal{Q}_{j-s,u}^{2,\alpha}}b_{Q}.$$
	Set
	$$E_{\epsilon\alpha/2}^{\imath,u,s}=\Big\{x\in\mathbb{R}^{d}:\sup_{k\in\mathbb{Z}}\Big|\sum_{j\geq k}H_{j}*B^{\imath,u}_{j-s}(x)\Big|>\epsilon\alpha/2\Big\},\ \imath=1,2.$$
	Then we split $E_{\epsilon\alpha}^{s}$ in two parts,
	$$E_{\epsilon\alpha}^{s}\subset E_{\epsilon\alpha/2}^{1,u,s}\cup E_{\epsilon\alpha/2}^{2,u,s}.$$
	Hence we obtain
	$$\begin{array}{ll}
		&\displaystyle\int_{E_{\epsilon\alpha}^{s}}\min\{v(x),u\} dx\lesssim\displaystyle\int_{E_{\epsilon\alpha/2}^{1,u,s}}v(x)dx+u\int_{E_{\epsilon\alpha/2}^{2,u,s}}dx\\
		&\qquad\qquad\qquad\qquad\lesssim\displaystyle\sum\limits_{j\in\mathbb{Z}}\sum\limits_{Q\in \mathcal{Q}_{j-s,u}^{1,\alpha}}|Q|\inf_{z\in Q}Mv(x)
		+\epsilon^{-2}2^{-\delta s}u\sum\limits_{j\in\mathbb{Z}}\sum_{Q\in \mathcal{Q}_{j-s,u}^{2,\alpha}}|Q|.
	\end{array}$$
	This yields (3.5). By Fubini theorem and the following identity
	$$\int_{0}^{\infty}\min\{v(x),u\}u^{-1+\theta}\frac{du}{u}=\theta^{-1}(1-\theta)^{-1}v(x)^{\theta},\eqno(3.6)$$
	we obtain that
	$$\int_{E_{\epsilon\alpha}^{s}}v(x)^{\theta}dx=\theta(1-\theta)\epsilon^{-2}\int_{0}^{\infty}\sum\limits_{Q\in \mathcal{Q}_{\alpha}}|Q|\min\Big\{u2^{-\delta s},\inf_{z\in Q}Mv(z)\Big\}u^{-1+\theta}\frac{du}{u},$$
	which leads to
	$$\frac{1}{\epsilon^{2}}\sum\limits_{Q\in\mathcal{Q}_{\alpha}}|Q|2^{-\delta(1-\theta)s}(\inf_{z\in Q}Mv(z))^{\theta}
	\lesssim_d\frac{2^{-\delta(1-\theta)s}}{\epsilon^{2}\alpha}\sum\limits_{Q\in \mathcal{Q}_{\alpha}}\int_{Q}|b_{Q}(y)|(\inf_{z\in Q}Mv(z))^{\theta}dy.$$
	Taking $v=w^{\frac{1}{\theta}}$ in the above inequality, one gets
	$$w(E_{\epsilon\alpha}^{s})\lesssim_d\frac{2^{-\delta(1-\theta)s}}{\epsilon^{2}\alpha}\sum\limits_{Q\in \mathcal{Q}_{\alpha}}
	\int_{Q}|b_{Q}(y)|M_{\frac{1}{\theta}}w(y)dy\lesssim 2^{-\delta(1-\theta)s}\epsilon^{-2}\alpha^{-1} \|f\|_{L^1(M_{{\frac{1}{\theta}}}w)}.\eqno(3.7)$$
	To get the better bound, we need to split the set $\{s:s\geq 200\}$ in 
	two parts. Since $w\in A_{1}$, then we get by Lemma 
	\ref{lem2.1} that for any cube $Q$ and 
	$\theta^{-1}\in [1,1+(\tau_{d}[w]_{A_{\infty}})^{-1}]$, $M_{\frac{1}{\theta}}w
	\leq 2Mw\leq [w]_{A_{1}}w$, where the constant $\tau_{d}$ is given in 
	Lemma \ref{lem2.1}. Now we set
	$$\theta=\frac{\tau_{d}[w]_{A_{\infty}}}{1+\tau_{d}[w]_{A_{\infty}}},\ \ s_{0}(\delta)=\Big[\frac{2}{\delta(1-\theta)}\log_{2}([w]_{A_{\infty}}+1)\Big],$$
	$$\epsilon_{s}=C\delta(1-\theta)2^{-\frac{\delta(1-\theta)(s-s_{0}(\delta))}{3}},\ \mathrm{\ so\ that }\ \sum\limits_{s=s_{0}(\delta)+1}^\infty\epsilon_{s}=1/2.$$
	Write
	$$\Big\{x\in\mathbb{R}^{d}:\sum\limits_{s=200}^\infty\sup\limits_{k\in\mathbb{Z}}\Big|\sum\limits_{j\geq k}H_{j}*B_{j-s}(x)\Big|>\alpha\Big\}
	\subset G_{\alpha/2}^{s} \cup\bigcup_{s=s_{0}(\delta)+1}^\infty E_{\epsilon_{s}\alpha}^{s},\eqno(3.8)$$
	where
	\begin{align*}
		G_{\alpha/2}^{s}&:=\Big\{x\in\mathbb{R}^{d}:\sum\limits_{s=200}^{s_{0}(\delta)}\sup\limits_{k\in\mathbb{Z}}\Big|\sum\limits_{j\geq k }H_{j}*B_{j-s}(x)\Big|>\alpha/2\Big\},\\
		E_{\epsilon_{s}\alpha}^{s}&:=\Big\{x\in\mathbb{R}^{d}:\sup\limits_{k\in\mathbb{Z}}\Big|\sum\limits_{j\geq k}H_{j}*B_{j-s}(x)\Big|>\epsilon_{s}\alpha\Big\},\ \ s\geq s_0(\delta)+1.
	\end{align*}
	
	Since $w\in A_{1}$, then
	$$w(G_{\alpha/2}^{s})\lesssim\alpha^{-1}\Big\|\sum\limits_{s=200}^{s_0(\delta)}\sup\limits_{k\in\mathbb{Z}}\Big|\sum\limits_{j\geq k}H_{j}*B_{j-s}\Big|\Big\|_{L^1(w)}
	\lesssim_d s_{0}(\delta)[w]_{A_{1}}\alpha^{-1}\|f\|_{L^1(w)},$$
	$$\sum\limits_{s=s_0(\delta)+1}^\infty w(E^{s}_{\epsilon_{s}\alpha})
	\lesssim_d\alpha^{-1}\sum_{s=s_{0}(\delta)+1}^\infty(\epsilon_{s})^{-2}2^{-\delta(1-\theta)s}\|f\|_{L^1(M_{{\frac{1}{\theta}}}w)}.$$
	Thus, one obtains
	$$\begin{array}{ll}
		&\displaystyle w\Big(\Big\{x\in\mathbb{R}^{d}:\sum\limits_{s=200}^\infty\sup\limits_{k\in\mathbb{Z}}\Big|\sum_{j\geq k}H_{j}*B_{j-s}(x)\Big|\Big)>\alpha\Big\}\Big|\\
		&\lesssim_{d}\displaystyle\alpha^{-1} \Big(s_{0}(\delta)+\frac{2^{-s_{0}(\delta)\delta(1-\theta)}}{(1-\theta)^{3}}\Big)\|f\|_{L^1(M_{{\frac{1}{\theta}}}w)}\\
		&\lesssim_d[w]_{A_{1}}[w]_{A_{\infty}} \log_{2}([w]_{A_{\infty}}+1)\alpha^{-1}\|f\|_{L^1(w)}.
	\end{array}$$
	This proves (3.3) and completes the proof of Theorem \ref{thm1.1}. $\hfill\Box$
	
	\bigskip

	\section{Proof of Theorem \ref{thm1.2}}\label{S4}
	
	This section aims to prove Theorem \ref{thm1.2}. In what follows, we assume $\|\Omega\|_{L^\infty(\mathbb{S}^{d-1})}=1$. 
	Let $f\in L^1(w)$, $w\in A_{1}$. By density, 
	we may assume $f\in C_c^\infty(\mathbb{R}^{d})$. For (1.14), it suffices to show that
	$$\sup_{\lambda>0}\Big\|\lambda\mathcal{N}_{\lambda}\Big(\Big\{\sum\limits_{j\in\mathbb{Z}}(H_{j}\one_{B^{c}_{\varepsilon}})*g\Big\}_{\varepsilon\in \mathbb{R}^{+}}\Big)^{1/2}\Big\|_{L^2(w)}\lesssim_{d,w}\|g\|_{L^2(w)},\eqno(4.1)$$
	and for each $\lambda>0$,
	$$w\Big(\Big\{x\notin E:\lambda\mathcal{N}_{\lambda}\Big(\Big\{\sum_{s=200}^\infty\sum\limits_{j\in\mathbb{Z}}(H_{j}\one_{B^{c}_{2^{k}}})*B_{j-s}\Big\}_{k\in\mathbb{Z}}\Big)^{1/2}(x)>\alpha\Big\}\Big)
	\lesssim_{d,w}\alpha^{-1}\|f\|_{L^1(w)},\eqno(4.2)$$
	for all $\alpha>0$.
	The estimate (4.1) follows from Theorem B. Now we prove (4.2). For 
	convenience, we denote
	$$O_{k}=\{x\in\mathbb{R}^{d}:|x|\leq2^{k}\},\ \ \ A_{k}=\{x\in\mathbb{R}^{d}:2^{k-1}<|x|\leq 2^{k}\},\ \ \ k\in\mathbb{Z}.$$
	By the support of $H_{j}$ and (1.7), we write
	$$\begin{array}{ll}
		&\displaystyle \lambda \mathcal{N}_{\lambda}\Big(\Big\{\sum_{s=200}^\infty\sum\limits_{j\in\mathbb{Z}}(H_{j}\one_{B^{c}_{2^{k}}})*B_{j-s}\Big\}_{k\in\mathbb{Z}}\Big)^{1/2}(x)\\
		&\leq\displaystyle \lambda \mathcal{N}_{\lambda/2}\Big(\Big\{\sum\limits_{s\geq 200}\sum\limits_{j\geq k+2}H_{j}*B_{j-s}\Big\}_{k\in\mathbb{Z}}\Big)^{1/2}(x)\\
		&\qquad+\displaystyle\lambda \mathcal{N}_{\lambda/2}\Big(\Big\{\sum\limits_{s\geq 200}(H_{k+1}\one_{A_{k+1}})*B_{k+1-s}\Big\}_{k\in\mathbb{Z}}\Big)^{1/2}(x)\\
		&\qquad\qquad\lesssim\displaystyle\sum\limits_{s=200}^\infty V_2\Big(\Big\{\sum_{j\geq k}H_{j}*B_{j-s}\Big\}_{k\in\mathbb{Z}}\Big)
		+\sum\limits_{s=200}^\infty V_2\Big(\Big\{H_{k}\one_{A_{k}})*B_{k-s}\Big\}_{k\in\mathbb{Z}}\Big)\\
		&\lesssim\displaystyle\sum\limits_{s=200}^\infty V_2\Big(\Big\{\sum_{j\geq k}H_{j}*B_{j-s}\Big\}_{k\in\mathbb{Z}}\Big)
		+\sum\limits_{s=200}^\infty\Big(\sum\limits_{j\in\mathbb{Z}}|(H_{j}\one_{A_{j}})*B_{j-s}|^{2}\Big)^{1/2}.
	\end{array}$$
	By Proposition \ref{pro2.1}, there exists a constant $\delta>0$ such that
	$$\Big|\Big\{x\in\mathbb{R}^{d}:V_2\Big(\Big\{\sum\limits_{j\geq k}H_{j}*B_{j-s}\Big\}_{k\in\mathbb{Z}}\Big)(x)>\epsilon\alpha\Big\}\Big|
	\lesssim_d\frac{2^{-\delta s}}{\epsilon^{2}}\sum\limits_{j\in\mathbb{Z}}\sum\limits_{Q\in \mathcal{Q}^{\alpha}_{j-s}}|Q|.\eqno(4.3)$$
	Next we shall prove that  there exists a constant $\delta_{1}>0$ such that
	$$\Big|\Big\{x\in\mathbb{R}^{d}:\Big(\sum\limits_{j\in\mathbb{Z}}|(H_{j}\one_{A_{j}})*B_{j-s}(x)|^{2}\Big)^{1/2}>\epsilon\alpha\Big\}\Big|
	\lesssim_d\frac{2^{-\delta_{1}s}}{\epsilon^{2}}\sum\limits_{j\in\mathbb{Z}}\sum\limits_{Q\in \mathcal{Q}^{\alpha}_{j-s}}|Q|.\eqno(4.4)$$
	We use Khintchine inequality and Seeger's estimates (see \cite[Lemmas 2.1-2.2]{See1}
	and \cite[Lemma A.1]{BS}) to prove (4.4). Let $\psi\in C_{c}^{\infty}(\mathbb{R})$ 
	be a bump function near the origin, that is, supported on the interval $[-2^{-10},
	2^{-10}]$ and $\int_{\mathbb{R}}\psi(x)dx=1$. Set $\psi_{j-\iota s}(x)=
	2^{-j+\iota s}\psi(2^{-j+\iota s}x)$. Write
	$$(H_{j}\one_{A_{j}})(x)={H}^{\iota}_{j}(x)+S^{\iota}_{j}(x),$$
	where 
	$${H}^{\iota}_{j}(x)={H}^{\iota}_{j}(r,\theta)=(H_{j}\one_{A_{j}})(\cdot,\theta)*\psi_{j-\iota s}(r),$$
	$$S^{\iota}_{j}(x)=S^{\iota}_{j}(r,\theta)=(H_{j}\one_{A_{j}})(r,\theta)-{H}^{\iota}_{j}(r,\theta).$$
	Let $\bar{\gamma}(r)=\phi(2r)-\phi(r)$, where $\phi$ is given in (1.17). Then
	$$H^{\iota}_{j}(r,\theta)=\Omega(\theta)\int_{\mathbb{R}}\rho^{-d}\bar{\gamma}_{j}(\rho)\one_{A_{j}}(\rho\theta)\psi_{j-\iota s}(r-\rho)d\rho,$$
	$$S^{\iota}_{j}(r,\theta)=\Omega(\theta)\Big(r^{-d}\bar{\gamma}_{j}(r)\one_{A_{j}}(r\theta)-\int_{\mathbb{R}}\rho^{-d}\bar{\gamma}_{j}(\rho)\one_{A_{j}}(\rho\theta)\psi_{j-\iota s}(r-\rho)d\rho\Big).$$
	Note that
	$$\mathfrak{M}_{N}:=\sup\limits_{j\in\mathbb{Z}}\sup\limits_{r>0}\sup\limits_{\theta\in\mathbb{S}^{d-1}}
	\sup_{0\leq l\leq N} r^{d+l}\big|\partial_{r}^{l}{H}^{\iota}_{j}(\cdot,\theta)\big|\lesssim_{N}2^{\iota (N+1)s}\|\Omega\|_{L^\infty(\mathbb{S}^{d-1})}.\eqno(4.5)$$
	Let $\Gamma_{\nu}^{s}$ and $P^{s}_{\nu}$ be defined in (2.2) and (2.3). Define 
	${H}^{\iota,s}_{j,\nu}(x)={H}^{\iota}_{j}(x)\Gamma_{\nu}^{s}(x)$. By the 
	arguments similar to those used to derive \cite[Lemmas 2.1-2.2]{See1}, one 
	obtains
	$$\Big\|\sum\limits_{j\in\mathbb{Z}}\epsilon_{j}
	\Big(\sum\limits_{\nu\in \Lambda_{s}}P_{\nu}^{s}*H^{\iota,s}_{j,\nu}\Big)*B_{j-s}\Big\|^{2}_{2}\lesssim_d\mathfrak{M}_{0}^{2}2^{-s\gamma}\alpha\sum\limits_{j\in\mathbb{Z}}\sum\limits_{Q\in \mathcal{Q}^{\alpha}_{j-s}}\|b_{Q}\|_{1},\eqno(4.6)$$
	where $\epsilon_{j}\in\{-1,1\}$ and for each $\varepsilon_0\in(0,1)$, 
	$N_{1}\geq d+1$,
	$$\Big\|\sum\limits_{j\in\mathbb{Z}}\epsilon_{j}\sum\limits_{\nu\in\Lambda_{s}}
	(H^{\iota,s}_{j,\nu}-P_{\nu}^{s}*H^{\iota,s}_{j,\nu})*B_{j-s}\Big\|_{1}\lesssim_d(\mathfrak{M}_{0}2^{-s\eta_{1}}+\mathfrak{M}_{N_{1}}2^{-s\eta_{2}})\sum\limits_{j\in\mathbb{Z}}\sum\limits_{Q\in \mathcal{Q}^{\alpha}_{j-s}}\|b_{Q}\|_{1},\eqno(4.7)$$
	where $\eta_{1}=1-\varepsilon_{0}$, $\eta_{2}=(\varepsilon_{0}-\gamma)N_{1}-
	d\varepsilon_{0}-d\gamma$, $\epsilon_{j}\in\{-1,1\}$. It follows from 
	(4.5)--(4.7) and Khintchine inequality that
	$$\Big|\Big\{x\in\mathbb{R}^{d}:\Big(\sum\limits_{j\in\mathbb{Z}}|{H}^{\iota}_{j}*B_{j-s}(x)|^{2}\Big)^{1/2}>\epsilon\alpha/2\Big\}\Big|
	\lesssim_d\frac{2^{-\delta_{1}s}}{\epsilon^{2}}\sum\limits_{j\in\mathbb{Z}}\sum\limits_{Q\in \mathcal{Q}^{\alpha}_{j-s}}|Q|.\eqno(4.8)$$
	For $S^{\iota}_{j}$, if $s\geq 3\iota^{-1}$, we get by mean value theorem that
	$$\begin{array}{ll}
		&\displaystyle\Big\|\Big(\sum\limits_{j\in\mathbb{Z}}|S^{\iota}_{j}*B_{j-s}|^{2}\Big)^{1/2}\Big\|_{1}\\
		&\leq\displaystyle\sum\limits_{j\in\mathbb{Z}}\int_{\mathbb{S}^{d-1}}|\Omega(\theta)|\int_{0}^{\infty}\int_{\mathbb{R}}|\Delta_{-\varrho}(r^{-d}\bar{\gamma}_{j}(r)\one_{A_{j}}(r\theta))|
		\psi_{j-\iota s}(\rho)d\rho \|B_{j-s}\|_{1}r^{d-1}drd\sigma(\theta)\\
		&\lesssim\displaystyle\sum\limits_{j\in\mathbb{Z}}\|B_{j-s}\|_{1}\int_{\mathbb{R}}\Big(\int_{0}^{\infty}|\Delta_{-\rho}(r^{-d}\bar{\gamma}_{j}(r)\one_{A_{j}}(r\theta))|r^{d-1}dr\Big)\psi_{j-\iota s}(\rho)d\rho\\
		&\lesssim\displaystyle \min\{1,2^{-\iota s+3}\}\sum\limits_{j\in\mathbb{Z}}\|B_{j-s}\|_{1},
	\end{array}$$
	where $\Delta_{-\varrho}h=h(\cdot-\rho)-h$. This together with (4.8) yields (4.4).
	
	In addition, for each non-negative function $v$, we also have the following 
	trivial weighted estimates:
	$$\Big\|V_2\Big(\Big\{\sum\limits_{j\geq k}H_{j}*B_{j-s}\Big\}_{k\in\mathbb{Z}}\Big)\Big\|_{L^{1}(v)}
	\lesssim_d\sum\limits_{j\in \Z}\sum\limits_{\substack{Q\in \mathcal{Q}_{j-s}^{\alpha}}}\|b_{Q}\|_{1}\inf_{z\in Q}Mv(z),\eqno(4.9)$$
	$$\Big\|\Big(\sum\limits_{j\in\mathbb{Z}}|(H_{j}\one_{A_{j}})*B_{j-s}|^{2}\Big)^{1/2}\Big\|_{L^{1}(v)}
	\lesssim_d\sum\limits_{j\in \Z}\sum\limits_{\substack{Q\in \mathcal{Q}_{j-s}^{\alpha}}}\|b_{Q}\|_{1}\inf_{z\in Q}Mv(z).\eqno(4.10)$$
	The following arguments are similar to those used to derive Theorem \ref{thm1.1}. Set $\delta_{2}=\min\{\delta,\delta_{1}\}$. Define
	$$\mathcal{E}_{\epsilon\alpha}^{s,1}=\Big\{x\in\mathbb{R}^{d}:V_2\Big(\Big\{\sum\limits_{j\geq k}H_{j}*B_{j-s}\Big\}_{k\in\mathbb{Z}}\Big)(x)>\epsilon\alpha\Big\},$$
	$$\mathcal{E}_{\epsilon\alpha}^{s,2}=\Big\{x\in\mathbb{R}^{d}:\Big(\sum\limits_{j\in\mathbb{Z}}|(H_{j}\one_{A_{j}})*B_{j-s}(x)|^{2}\Big)^{1/2}>\epsilon\alpha\Big\}.$$
	Then we define
	$$\mathcal{E}_{\epsilon\alpha/2}^{\imath,u,s,1}=\Big\{x\in\mathbb{R}^{d}:V_2\Big(\Big\{\sum\limits_{j\geq k}H_{j}*B^{\imath,\alpha,s}_{j-s,u}\Big\}_{k\in\mathbb{Z}}\Big)(x)>\epsilon\alpha/2\Big\},\ \imath=1,2,$$
	$$\mathcal{E}_{\epsilon\alpha/2}^{\imath,u,s,2}=\Big\{x\in\mathbb{R}^{d}:\Big(\sum\limits_{j\in\mathbb{Z}}|(H_{j}\one_{A_{j}})*B^{\imath,\alpha,s}_{j-s,u}(x)|^{2}\Big)^{1/2}>\epsilon\alpha/2\Big\},\ \imath=1,2,$$
		where $B^{1,\alpha,s}_{j-s,u}$, $B^{2,\alpha,s}_{j-s,u}$ are defined analogously to that in the proof of Theorem \ref{thm1.1}, that is, for each $u>0$ and $j\in\Z$, let $\mathcal{Q}^{1,\alpha,
		s}_{j-s,u}$ and $\mathcal{Q}^{2,\alpha,s}_{j-s,u}$ be given in the 
	proof of Theorem \ref{thm1.1} with $\delta$ replaced by $\delta_2$.  Then we define
	$$B^{1,\alpha,s}_{j-s,u}=\sum\limits_{Q\in \mathcal{Q}^{1,\alpha,s}_{j-s,u}}b_{Q},\ \ \ \ B^{2,\alpha,s}_{j-s,u}=\sum\limits_{Q\in \mathcal{Q}^{2,\alpha,s}_{j-s,u}}b_{Q}.$$
	Clearly, $\mathcal{E}_{\epsilon\alpha}^{s,i}\subset \mathcal{E}_{\epsilon
		\alpha/2}^{1,u,s,i}\cup \mathcal{E}_{\epsilon\alpha/2}^{2,u,s,i}$, $i=1,2$. 
	It follows from (4.3), (4.4), (4.9) and (4.10) that there exists a constant 
	$\delta_2>0$ such that for each $i=1,2$,
	$$\begin{array}{ll}
		&\displaystyle\int_{\mathcal{E}_{\epsilon\alpha}^{s,i}}\min\{v(x),u\} dx\lesssim\displaystyle\int_{\mathcal{E}_{\epsilon\alpha/2}^{1,u,s,i}}v(x)dx+u\int_{\mathcal{E}_{\epsilon\alpha/2}^{2,u,s,i}}dx\\
		&\qquad\qquad\qquad\qquad\lesssim_d\displaystyle\epsilon^{-2}\sum_{Q\in \mathcal{Q}_{\alpha}}|Q|\min\{u2^{-\delta_{2}s},\inf_{z\in Q}Mv(z)\}.
	\end{array}$$
	We get by (3.6) that for each $i=1,2$,
	$$\int_{\mathcal{E}_{\epsilon\alpha}^{s,i}}v(x)^{\theta}dx=\frac{1}{\epsilon^{2}}\sum\limits_{Q\in \mathcal{Q}_{\alpha}}|Q|2^{-\delta_{2}(1-\theta)s}(\inf_{z\in Q}Mv(z))^{\theta}.$$
	By the above inequality (with $v^{\theta}$ replaced by $w$), one gets
	$$\begin{array}{ll}
		&w(\mathcal{E}^{s,1}_{\epsilon\alpha})+w(\mathcal{E}^{s,2}_{\epsilon\alpha})
		\lesssim_d\displaystyle\frac{2^{-\delta_{2}(1-\theta)s}}{\epsilon^{2}}\sum\limits_{\substack{Q\in \mathcal{Q}_{\alpha}}}|Q|\inf_{z\in Q}M_{\frac{1}{\theta}}w(z)\\
		&\qquad\qquad\qquad\quad\lesssim_d\displaystyle\frac{2^{-\delta_{2}(1-\theta)s}}{\epsilon^{2}\alpha}\|f\|_{L^1(M_{\frac{1}{\theta}}w)}.
	\end{array}\eqno(4.11)$$
	Let $\theta$, $s_0(\delta_{2})$, $\epsilon_{s}$ be given in the proof of Theorem \ref{thm1.1} with $\delta$ replaced by $\delta_{2}$. By piegonhole principle, we have
	$$\begin{array}{ll}
		&\displaystyle\Big\{x\in\mathbb{R}^{d}:\lambda\mathcal{N}_{\lambda}\Big(\Big\{\sum\limits_{s=200}^\infty\sum_{j\geq k}H_{j}*B_{j-s}\Big\}_{k\in\mathbb{Z}}\Big)^{1/2}(x)>2\alpha\Big\}\\
		&\qquad\subset\displaystyle\Big\{x\in\mathbb{R}^{d}:\sum\limits_{s=200}^{s_{0}(\delta_2)}V_2\Big(\Big\{\sum\limits_{j\geq k}H_{j}*B_{j-s}\Big\}_{k\in\mathbb{Z}}\Big)(x)>\alpha/2\Big\}\\
		&\qquad\qquad\cup\displaystyle\bigcup\limits_{s\geq s_{0}+1}\Big\{x:V_2\Big(\Big\{\sum\limits_{j\geq k}H_{j}*B_{j-s}\Big\}_{k\in\mathbb{Z}}\Big)(x)>\epsilon_{s}\alpha\Big\}\\
		&=:\displaystyle\mathcal{G}_{\alpha/2}^{s,1}\cup \bigcup_{s\geq s_{0}+1} \mathcal{E}^{s,1}_{\epsilon_{s}\alpha}.
	\end{array}\eqno(4.12)$$
	Moreover, we have
	$$\begin{array}{ll}
		&\displaystyle\Big\{x\in\mathbb{R}^{d}:\lambda \mathcal{N}_{\lambda}\Big(\Big\{\sum\limits_{s\geq 200}(H_{k}\one_{A_{k}})*B_{k-s}\Big\}_{k\in\mathbb{Z}}\Big)^{1/2}(x)>4\alpha\Big\}\\
		&\qquad\subset\displaystyle\Big\{x\in\mathbb{R}^{d}:\sum\limits_{s=200}^{s_{0}(\delta_2)}\Big(\sum\limits_{j\in\mathbb{Z}}|(H_{j}\one_{A_{j}})*B_{j-s}(x)|^{2}\Big)^{1/2}>\alpha/2\Big\}\\
		&\qquad\qquad\cup\displaystyle\bigcup_{s\geq s_0+1}\Big\{x\in\mathbb{R}^{d}:\Big(\sum\limits_{j\in\mathbb{Z}}|(H_{j}\one_{A_{j}})*B_{j-s}(x)|^{2}\Big)^{1/2}>\epsilon_{s}\alpha\Big\}\\
		&=:\displaystyle\mathcal{G}_{\alpha/2}^{s,2}\cup \bigcup_{s\geq s_{0}+1} \mathcal{E}^{s,2}_{\epsilon_{s}\alpha}.
	\end{array}\eqno(4.13)$$
	Therefore, by (4.11)--(4.13) and Chebyshev inequality, we obtain
	$$\begin{array}{ll}
		&\displaystyle w\Big(\Big\{x\notin E:\lambda\mathcal{N}_{\lambda}\Big(\Big\{\sum\limits_{s\geq 200}\sum_{j\in\mathbb{Z}}(H_{j}\one_{O^{c}_{k}})*B_{j-s}\Big\}_{k\in\mathbb{Z}}\Big)^{1/2}(x)>10\alpha\Big\}\Big)\\
		&\qquad\leq\displaystyle w\Big(\Big\{x\notin E:\lambda \mathcal{N}_{{\lambda}/{2}}\Big(\Big\{\sum\limits_{s\geq 200}\sum\limits_{j\geq k}H_{j}*B_{j-s}\Big\}_{k\in\mathbb{Z}}\Big)^{1/2}(x)>2\alpha\Big\}\Big)\\
		&\qquad\qquad+\displaystyle w\Big(\Big\{x\notin E:\lambda \mathcal{N}_{{\lambda}/{2}}\Big(\Big\{\sum\limits_{s\geq 200}(H_{k}\one_{A_{k}})*B_{k-s}\Big\}_{k\in\mathbb{Z}}\Big)^{1/2}(x)>4\alpha\Big\}\Big)\\
		&\qquad\qquad\qquad\leq\displaystyle\sum\limits_{i=1}^{2}w(\mathcal{G}_{\alpha/2}^{s,i})+\sum\limits_{i=1}^{2}\sum\limits_{s\geq s_{0}+1} w(E^{s,i}_{\epsilon_{s}\alpha})\\
		&\lesssim_d\displaystyle\alpha^{-1}\Big(s_{0}(\delta_2)+\sum\limits_{s\geq s_{0}(\delta_2)+1}(\epsilon_{s})^{-2}2^{-\delta_{2}(1-\theta)s}\Big)\|f\|_{L^1(M_{{\frac{1}{\theta}}}w)}.
	\end{array}$$
	Since $w\in A_1$, applying Lemma \ref{lem2.1}, we have
	$$\begin{array}{ll}
		&\displaystyle w\Big(\Big\{x\notin E:\lambda \mathcal{N}_{\lambda}\Big(\Big\{\sum\limits_{s\geq 200}\sum\limits_{j\in\mathbb{Z}}(H_{j}\one_{B^{c}_{2^{k}}})*B_{j-s}\Big\}_{k\in\mathbb{Z}}\Big)^{1/2}(x)>\alpha\Big\}\Big)\\
		&\lesssim_d[w]_{A_{1}}[w]_{A_{\infty}} \log_{2}([w]_{A_{\infty}}+1)\alpha^{-1}\|f\|_{L^1(w)},
	\end{array}$$
	which proves (4.2) and completes the proof of (1.14). $\hfill\Box$
	
	\bigskip
	
	\section{Proof of Theorem \ref{thm1.3}}\label{S5}
	
	In this section we prove Theorem \ref{thm1.3}.  In what follows, we assume $\|\Omega\|_{L^\infty(\mathbb{S}^{d-1})}=1$. 
	Let $f\in L^1(w)$, $w\in A_1$. By density, 
	we may assume $f\in C_c^\infty(\mathbb{R}^{d})$. For (1.18), it is sufficient to prove that
	$$\sup_{\lambda>0}\Big\|\lambda \mathcal{N}_{\lambda}\Big(\Big\{\sum\limits_{j\in\mathbb{Z}}(H_{j}\phi_{\varepsilon})*g\Big\}_{\varepsilon\in\mathbb{R}^{+}}\Big)^{1/2}\Big\|_{L^2(w)}
	\lesssim_{d,w}\|g\|_{L^2(w)},\eqno(5.1)$$
	and for each $\lambda>0$,
	$$w\Big(\Big\{x\notin E:\lambda\mathcal{N}_{1.001\lambda}\Big(\Big\{\sum\limits_{j\in\mathbb{Z}}(H_{j}\phi_{\varepsilon})*b\Big\}_{\varepsilon\in\mathbb{R}^{+}}\Big)^{1/2}(x)>200\alpha\Big\}\Big)
	\lesssim_{d,w}\alpha^{-1}\|f\|_{L^1(w)},\eqno(5.2)$$
	for all $\alpha>0$, where $(H_{j}\phi_{\varepsilon})(x)=H_{j}(x)\phi_{\varepsilon}(|x|)$. 
	The estimate (5.1) follows from (4.1). Indeed, observe that
	$$\sum\limits_{j\in\mathbb{Z}}(H_{j}\phi_{\varepsilon})*g=\int_{0}^{\infty}\phi'(\omega)\Big(\sum\limits_{j\in\mathbb{Z}}(H_{j}\one_{B^{c}_{\varepsilon\omega}})*g\Big)d\omega,$$
	see \cite[Remark 3.2]{HLP}. This together with the fact that $J^{2,2}_{2}(w)$ 
	has an equivalence (semi)norm(see Lemma 2.1 and Corollary 2.2, \cite{MSZ}) 
	and	Minkowski inequality yield (5.1).
	
	Next we prove (5.2). In view of (1.7) and (1.9), for each $x \notin E$, 
	there exists a constant $C>0$ independent of $\lambda$ such that
	$$\begin{array}{ll}
		&\displaystyle\lambda\mathcal{N}_{1.001\lambda}\Big(\Big\{\sum\limits_{j\in\mathbb{Z}}(H_{j}\phi_{\varepsilon})*b\Big\}_{\varepsilon\in\mathbb{R}^{+}}\Big)^{1/2}(x)\\
		&\quad\leq\displaystyle\lambda N_{\lambda}\Big(\Big\{\sum\limits_{j\in\mathbb{Z}}(H_{j}\phi_{\varepsilon})*b\Big\}_{\varepsilon\in\mathbb{R}^{+}}\Big)^{1/2}(x)\\
		&\quad\quad\leq\displaystyle C\Big(\mathcal{S}_{2}\Big(\Big\{\sum\limits_{j\in\mathbb{Z}}(H_{j}\phi_{\varepsilon})*b\Big\}_{\varepsilon\in\mathbb{R}^{+}}\Big)(x)+\displaystyle\lambda N^{\mathrm{dyad}}_{\lambda/3}\Big(\Big\{\sum\limits_{j\in\mathbb{Z}}(H_{j}\phi_{\varepsilon})*b\Big\}_{\varepsilon\in\mathbb{R}^{+}}\Big)^{1/2}(x)\Big)\\
		&\quad\quad\quad\leq\displaystyle C\Big(\mathcal{S}_{2}\Big(\Big\{\sum\limits_{j\in\mathbb{Z}}(H_{j}\phi_{\varepsilon})*\Big(\sum_{s\geq200}B_{j-s}\Big)\Big\}_{\varepsilon\in\mathbb{R}^{+}}\Big)(x)\\
		&\quad\quad\quad\quad\quad+\displaystyle\lambda N_{\lambda/3}\Big(\Big\{\sum\limits_{j\geq k}H_{j}*b\Big\}_{k\in\Z}\Big)^{1/2}(x)\Big)\\
		&\leq\displaystyle C\Big(\sum_{s\geq200}\mathcal{S}_{2}\Big(\Big\{\sum\limits_{j\in\mathbb{Z}}(H_{j}\phi_{\varepsilon})*B_{j-s}\Big\}_{\varepsilon\in\mathbb{R}^{+}}\Big)(x)\\
		&\qquad\qquad\qquad+\displaystyle\lambda N_{\lambda/3}\Big(\Big\{\sum\limits_{s\geq 200}\sum\limits_{j\geq k}H_{j}*B_{j-s}\Big\}_{k\in\Z}\Big)^{1/2}(x)\Big).
	\end{array}\eqno(5.3)$$
	Now we shall establish the conclusion that there exists $\delta_{3}>0$ such that
	$$\Big|\Big\{x\in\mathbb{R}^{d}:\mathcal{S}_{2}\Big(\Big\{\sum\limits_{j\in\mathbb{Z}}(H_{j}\phi_{\varepsilon})*B_{j-s}\Big\}_{\varepsilon\in\mathbb{R}^{+}}\Big)>\epsilon\alpha\Big\}\Big|
	\lesssim_d\displaystyle\frac{2^{-\delta_{3}s}}{\alpha\epsilon^{2}}\sum\limits_{j\in\mathbb{Z}}\sum_{Q\in\mathcal{Q}_{j-s}^{\alpha}}\|b_{Q}\|_{1}.\eqno(5.4)$$
	Let $H_{k,t,e}(r,\theta)=H_{k}(r\theta)(\phi({2^{-k+1-e}}t^{-1}r)-
	\phi(2^{-k-e}r))$. Then we have
	$$\begin{array}{ll}
		&\displaystyle\mathcal{S}_{2}\Big(\Big\{\sum\limits_{j\in\mathbb{Z}}(H_{j}\phi_{\varepsilon})*B_{j-s}\Big\}_{\varepsilon\in\mathbb{R}^{+}}\Big)(x)\\
		&\quad=\displaystyle\Big(\sum\limits_{k\in\mathbb{Z}}\Big(V_2\Big(\Big\{\sum\limits_{j\in\mathbb{Z}}(H_{j}\phi_\varepsilon)*B_{j-s}\Big\}_{\varepsilon\in[2^k,2^{k+1}]}\Big)(x)\Big)^{2}\Big)^{1/2}\\
		&\quad\quad\leq\displaystyle\Big(\sum\limits_{k\in\mathbb{Z}}\Big(V_2\Big(\Big\{\sum\limits_{j\in\mathbb{Z}}(H_{j}\phi_{2^{k}t}-H_{j}\phi_{2^{k+1}})*B_{j-s}\Big\}_{\varepsilon\in[1,2]}\Big)(x)\Big)^{2}\Big)^{1/2}\\
		&\quad\quad\quad=\displaystyle\Big(\sum\limits_{k\in\mathbb{Z}}\Big(V_2\Big(\Big\{\sum\limits_{j= k}^{k+2}(H_{j}\phi_{2^{k}t}-H_{j}\phi_{2^{k+1}})*B_{j-s}\Big\}_{\varepsilon\in[1,2]}\Big)(x)\Big)^{2}\Big)^{1/2}\\
		&\quad\quad\quad\quad\leq\displaystyle\sum\limits_{e=-1}^{1}\Big(\sum\limits_{k\in\mathbb{Z}}\Big(V_2\Big(\Big\{H_{k,t,e}*B_{k+e-s}\Big\}_{t\in[1,2]}\Big)(x)\Big)^{2}\Big)^{1/2}\\
		&=\displaystyle:\sum\limits_{e=-1}^{1}\mathrm{I}_{e}.
	\end{array}\eqno(5.5)$$
	We will consider the case $e=0$ and other cases are similar. From 
	\cite[(39)]{JSW} we see that
	$$\|a\|_{V_{\varrho}}\lesssim \|a\|_{\varrho}^{1-{1}/{\varrho}}\|a'\|_{\varrho}^{{1}/{\varrho}},\ \ \varrho>1.$$
	Taking $\varrho=2$, the above inequality gives that
	$$V_2(\{H_{k,t,0}*B_{k-s}\}_{t\in[1,2]})\lesssim\Big\|H_{k,t,0}*B_{k-s}\Big\|^{1/2}_{L^2([1,2])}\Big\|\frac{d}{dt}H_{k,t,0}*B_{k-s}\Big\|^{1/2}_{L^2([1,2])}.$$
	Therefore, by Cauchy--Schwarz inequality, we have
	$$\|\mathrm{I}_{0}\|_{2}\lesssim\Big\|\Big(\int_{1}^{2}\sum\limits_{k\in\mathbb{Z}}|H_{k,t,0}*B_{k-s}|^{2}
	\frac{dt}{t}\Big)^{1/2}\Big\|^{1/2}_{2}\Big\|\Big(\int_{1}^{2}\sum\limits_{k\in\mathbb{Z}}|{G}_{k,t,0}*B_{k-s}|^{2}
	\frac{dt}{t}\Big)^{1/2}\Big\|^{1/2}_{2},\eqno(5.6)$$
	where the kernel $G_{k,t,0}$ is defined by
	$$G_{k,t,0}(r,\theta)=-t\frac{d}{dt}H_{k,t,0}(r,\theta)=\Omega(\theta)r^{-d+1}\bar{\gamma}_{k}(r)(2^{k-1}t)^{-1}\phi'((2^{k-1}t)^{-1}r).$$
	Now we write
	$${H}_{j,t,0}=\sum\limits_{\nu\in \Lambda_{s}}P_{\nu}^{s}*H^{\nu,s}_{j,t,0}+\Big({H}_{j,t,0}-\sum_{\nu\in \Lambda_{s}}P_{\nu}^{s}*{H}^{\nu,s}_{j,t,0}\Big),$$
	where ${H}^{\nu,s}_{j,t,0}(x)={H}_{j,t,0}(x)\Gamma_{\nu}^{s}(x)$. Applying
	\cite[Lemma 3.3]{Lai1}, one gets that for each $\gamma\in(0,1)$,
	$$\Big\|\sum\limits_{\nu\in \Lambda_{s}}\sum\limits_{j\in\mathbb{Z}}\epsilon_{j}P_{\nu}^{s}*H^{\nu,s}_{j,t,0}*B_{j-s}\Big\|^{2}_{2}
	\lesssim_d\mathfrak{M}_{0,t,0}^{2}2^{-s\gamma}\alpha\sum\limits_{j\in\mathbb{Z}}\sum_{Q\in\mathcal{Q}_{j-s}^{\alpha}}\|b_{Q}\|_{1},$$
	where $\epsilon_{j}\in\{-1,1\}$ and
	$$\mathfrak{M}_{N,t,0}=\sup\limits_{j\in\mathbb{Z}}\sup\limits_{r>0}\sup\limits_{\theta\in\mathbb{S}^{d-1}}\sup_{0\leq l\leq N} r^{d+l}\big|\partial_{r}^{l}{H}_{j,t,0}(\cdot,\theta)\big|.$$
	By \cite[Lemma 3.4]{Lai1}, we have that for each $\varepsilon_0\in(0,1)$ and 
	$N_{1}\geq d+1$,
	$$\Big\|\sum\limits_{\nu\in \Lambda_{s}}\sum\limits_{j\in\mathbb{Z}}\epsilon_{j}(H^{\nu,s}_{j,t,0}-P_{\nu}^{s}*H^{\nu,s}_{j,t,0})*B_{j-s}\Big\|_{1}
	\lesssim_d(\mathfrak{M}_{0,t,0}2^{-s\eta_{1}}+\mathfrak{M}_{N_{1},t,0}2^{-s\eta_{2}})\sum\limits_{j\in\mathbb{Z}}\sum_{Q\in\mathcal{Q}_{j-s}^{\alpha}}\|b_{Q}\|_{1},$$
	where $\eta_{1}=1-\varepsilon_{0}$, $\eta_{2}=(\varepsilon_{0}-\gamma)N_{1}-
	d\varepsilon_{0}-d\gamma$. Applying \cite[Lemma 4.6]{Lai}, we get the trivial 
	bound
	$$\Big\|\sum\limits_{\nu\in\Lambda_{s}}\sum\limits_{j\in\mathbb{Z}}\epsilon_{j}(H^{\nu,s}_{j,t,0}-P_{\nu}^{s}*H^{\nu,s}_{j,t,0})*B_{j-s}\Big\|^{3}_{3}
	\lesssim_{d} 2^{2s\gamma(d-1)+3([\frac{d}{2}]+1)}\alpha^{2}\sum\limits_{j\in\mathbb{Z}}\sum_{Q\in\mathcal{Q}_{j-s}^{\alpha}}\|b_{Q}\|_{1}.$$
	Therefore, by the arguments of interpolation and the fact that $\mathfrak{M}_{N,t,0}
	\lesssim_{N} 1$,
	$$\Big\|\Big(\sum\limits_{j\in\mathbb{Z}}|H_{j,t,0}*B_{j-s}|^{2}\Big)^{1/2}\Big\|^{2}_{2}
	\lesssim_d(2^{-\gamma s}+2^{-{2\vartheta}_{1}s}+2^{-2\vartheta_{2} s})\alpha \sum\limits_{j\in\Z}\sum\limits_{Q\in \mathcal{Q}_{\alpha}}\|b_{Q}\|_{1},\eqno(5.7)$$
	where
	$$\vartheta_{1}=\frac{1}{4}(1-\varepsilon_{0})-\frac{3}{4}\gamma\Big(\frac{2}{3}(d-1)+[\frac{d}{2}]+1\Big),$$
	$$\vartheta_{2}=\frac{1}{4}\Big((\varepsilon_{0}-\gamma)N_{1}-(\gamma+\varepsilon_{0})d\Big)-\frac{3}{4}\gamma\Big(\frac{2}{3}(d-1)+[\frac{d}{2}]+1\Big).$$
	Now we choose $\iota\ll\gamma\ll \varepsilon_{0}\ll1$ and $N_{1}\geq d+1$ 
	such that $\vartheta_{1}$, $\vartheta_{2}>0$. Observe that
	$$\begin{array}{ll}
		&\displaystyle\Big\|\Big(\sum\limits_{j\in\mathbb{Z}}|{G}_{j,t,0}*B_{j-s}|^{2}\Big)^{1/2}\Big\|^{2}_{2}\\
		&\qquad\leq\displaystyle\sum\limits_{j\in\mathbb{Z}}\int_{\mathbb{R}^{d}}\int_{\mathbb{R}^{d}}|G_{j,t,0}*\widetilde{G}_{j,t,0}(z-y)||B_{j-s}(y)|dy|B_{j-s}(z)|dz\\
		&\qquad\qquad\lesssim_d\displaystyle (2^{j-2}t)^{-d}\sum\limits_{j\in\mathbb{Z}}\int_{\mathbb{R}^{d}}\int_{B(z,2^{j+2 }t)}|B_{j-s}(y)|dy|B_{j-s}(z)|dz\\
		&\qquad\qquad\qquad\lesssim_{d}\displaystyle(2^{j-2}t)^{-d}\sum\limits_{j\in\mathbb{Z}}\int_{\mathbb{R}^{d}}\Big(\sum\limits_{\substack{Q\in \mathcal{Q}_{\alpha}, L(Q)=j-s\\ Q\cap B(z, 2^{j+2}t)\neq\emptyset}}\int_{Q}|b_{Q}(y)|dy\Big)|B_{j-s}(z)|dz\\
		&\lesssim_{d}\displaystyle\alpha\sum\limits_{j\in\mathbb{Z}}\|B_{j-s}\|_{1}.
	\end{array}$$
	This together with (5.6) and (5.7) implies that
	$$\|\mathrm{I}_{0}\|^{2}_{2}\lesssim_d2^{-\delta_{4}s}\alpha\sum\limits_{j\in\Z}\sum\limits_{Q\in \mathcal{Q}_{j-s}^{\alpha}}\|b_{Q}\|_{1},\eqno(5.8)$$
	for some $\delta_4>0$. Similarly, one gets
	$$\|\mathrm{I}_{e}\|^{2}_{2}\lesssim_d2^{-\delta_{5}s}\alpha\sum\limits_{j\in\Z}\sum\limits_{Q\in \mathcal{Q}_{j-s}^{\alpha}}\|b_{Q}\|_{1},\ \ e=1,\,-1.\eqno(5.9)$$
	for some $\delta_5>0$. Now set $\delta_{3}=\min\{\delta, \delta_{4},\delta_{5}\}$. Then (5.4) follows from (5.5), (5.8), (5.9) and 
	Chebyshev inequality.
	
	In addition, it is not difficult to verify that for each non-negative 
	function $v$,
	$$\Big\|\mathcal{S}_{2}\Big(\Big\{\sum\limits_{j\in\mathbb{Z}}(H_{j}\phi_{\varepsilon})*B_{j-s}\Big\}_{\varepsilon\in\mathbb{R}^{+}}\Big)\Big\|_{L^{1}(v)}
	\lesssim_d\sum\limits_{j\in\mathbb{Z}}\sum_{Q\in\mathcal{Q}_{j-s}^{\alpha}}\|b_{Q}\|_{1}\inf_{z\in Q}Mv(z).\eqno(5.10)$$
	Let $\mathcal{E}^{s,1}_{\epsilon\alpha}$ be given in the proof of Theorem 
	\ref{thm1.2}. Set
	$$\mathcal{E}_{\epsilon\alpha}^{s,3}=\Big\{x\in\mathbb{R}^{d}:\mathcal{S}_{2}\Big(\Big\{\sum\limits_{j\in\mathbb{Z}}(H_{j}\phi_{\varepsilon})*B_{j-s}\Big\}_{\varepsilon\in\mathbb{R}^{+}}\Big)(x)>\epsilon\alpha\Big\},$$
	$$\mathcal{E}_{\epsilon\alpha/2}^{\imath,u,s,3}=\Big\{x\in\mathbb{R}^{d}:\mathcal{S}_{2}\Big(\Big\{\sum\limits_{j\in\mathbb{Z}}(H_{j}\phi_{\varepsilon})
	*B^{\imath,\alpha,s}_{j-s,u}\Big\}_{\varepsilon\in\mathbb{R}^{+}}\Big)(x)>\epsilon\alpha/2\Big\},\ \imath=1,2,$$
	where $B^{1,\alpha,s}_{j-s,u}$, $B^{2,\alpha,s}_{j-s,u}$, $\mathcal{Q}^{1,\alpha,
		s}_{j-s,u}$ and $\mathcal{Q}^{2,\alpha,s}_{j-s,u}$ are defined analogously to that in the proof of Theorem \ref{thm1.1} with $\delta$ replaced by $\delta_3$. Hence 
	  $\mathcal{E}_{\epsilon
		\alpha}^{s,3}\subset \mathcal{E}_{\epsilon\alpha/2}^{1,u,s,3}\cup
	\mathcal{E}_{\epsilon\alpha/2}^{2,u,s,3}$. In view of (5.4) and (5.10), one gets
	$$\begin{array}{ll}
		&\displaystyle\int_{\mathcal{E}_{\epsilon\alpha}^{s,3}}\min\{v(x),u\} dx\lesssim\displaystyle\int_{\mathcal{E}_{\epsilon\alpha/2}^{1,u,s,3}}v(x)dx+u\int_{\mathcal{E}_{\epsilon\alpha/2}^{2,u,s,3}}dx\\
		&\qquad\qquad\qquad\qquad\lesssim_d\displaystyle\epsilon^{-2}\sum\limits_{Q\in \mathcal{Q}_{\alpha}}|Q|\min\Big\{u2^{-\delta_{3} s},\inf_{z\in Q}Mv(z)\Big\}.
	\end{array}$$
	Applying (3.6), we have
	$$\int_{\mathcal{E}_{\epsilon\alpha}^{s,3}}v(x)^{\theta}dx=\frac{1}{\epsilon^{2}}\sum_{Q\in \mathcal{Q}_{\alpha}}|Q|2^{-\delta_{3}(1-\theta)s}(\inf_{z\in Q}Mv(z))^{\theta}.\eqno(5.11)$$
	 By (5.11) with $v=w^{1/\theta}$, we obtain
	$$w(\mathcal{E}^{s,3}_{\epsilon\alpha})\lesssim_d\displaystyle\frac{2^{-\delta_{3}(1-\theta)s}}{\epsilon^{2}}\sum\limits_{\substack{Q\in \mathcal{Q}_{\alpha}}}|Q|\inf_{z\in Q}M_{\frac{1}{\theta}}w(z)
	\lesssim_d \frac{2^{-\delta_{3}(1-\theta)s}}{\epsilon^{2}\alpha}  \|f\|_{L^1(M_{\frac{1}{\theta}}w)}.\eqno(5.12)$$
	On the other hand, one gets
	$$\lambda N_{\lambda/3}\Big(\Big\{\sum\limits_{s\geq 200}\sum\limits_{j\geq k}H_{j}*B_{j-s}\Big\}_{k\in\Z}\Big)^{1/2}(x)
	\lesssim \sum\limits_{s\geq 200}V_2\Big(\Big\{\sum\limits_{j\geq k}H_{j}*B_{j-s}\Big\}_{k\in\mathbb{Z}}\Big)(x).$$
	Let $\theta$, $s_0(\delta_{3})$ and $\epsilon_{s}$ be given in the proof of (1.14) with $\delta$ replaced by $\delta_{3}$. By the piegonhole principle, we have the inclusion that
	$$\begin{array}{ll}
		&\displaystyle\Big\{x\notin E:\lambda N_{\lambda/3}\Big(\Big\{\sum\limits_{s\geq 200}\sum\limits_{j\geq k}H_{j}*B_{j-s}\Big\}_{k \in\Z}\Big)^{1/2}(x)>12\alpha\Big\}\\
		&\qquad\displaystyle\subset\Big\{x\in\mathbb{R}^{d}:\sum\limits_{s=200}^{s_{0}(\delta_{3})}V_2\Big(\Big\{\sum\limits_{j\geq k}H_{j}*B_{j-s}\Big\}_{k\in\mathbb{Z}}\Big)(x)>\alpha/2\Big\}\\
		&\displaystyle\qquad\qquad\cup\bigcup_{s\geq s_{0}(\delta_{3})+1}\Big\{x\in\mathbb{R}^{d}:V_2\Big(\Big\{\sum\limits_{j\geq k}H_{j}*B_{j-s}\Big\}_{k\in\mathbb{Z}}\Big)(x)>\epsilon_{s}\alpha\Big\}\\
		&=:\displaystyle\mathcal{G}_{\alpha/2}^{s,1}\cup \bigcup_{s\geq s_{0}(\delta_{3})+1} \mathcal{E}^{s,1}_{\epsilon_{s}\alpha}.
	\end{array}\eqno(5.13)$$
Similarly, one gets
	$$\begin{array}{ll}
		&\displaystyle\Big\{x\notin E :\mathcal{S}_2\Big(\Big\{\sum\limits_{s\geq 200}\sum\limits_{j\in\mathbb{Z}}(H_{j}\phi_{\varepsilon})*B_{j-s}\Big\}_{\varepsilon\in\mathbb{R}^{+}}\Big)(x)>\alpha\Big\}\\
		&\qquad\displaystyle\subset\Big\{x\in\mathbb{R}^{d}:\sum\limits_{s=200}^{s_{0}(\delta_{3})}\mathcal{S}_2\Big(\Big\{\sum\limits_{j\in\mathbb{Z}}(H_{j}\phi_{\varepsilon})*B_{j-s}\Big\}_{\varepsilon\in\mathbb{R}^{+}}\Big)(x)>\alpha/2\Big\}\\
		&\qquad\qquad\displaystyle\cup\bigcup_{s\geq s_{0}(\delta_{3})+1}\Big\{x\in\mathbb{R}^{d}:\mathcal{S}_2\Big(\Big\{\sum\limits_{j\in\mathbb{Z}}(H_{j}\phi_{\varepsilon})*B_{j-s}\Big\}_{\varepsilon\in\mathbb{R}^{+}}\Big)(x)>\epsilon_{s}\alpha\Big\}\\
		&=:\displaystyle\mathcal{G}_{\alpha/2}^{s,3}\cup \bigcup_{s\geq s_{0}(\delta_{3})+1} \mathcal{E}^{s,3}_{\epsilon_{s}\alpha}.
	\end{array}\eqno(5.14)$$
	
	Therefore, by (4.11), (5.3), (5.12), (5.13), (5.14), the Chebyshev inequality 
	and the fact that $M_{{\frac{1}{\theta}}}w\leq 2Mw\lesssim [w]_{A_{1}}w$, we 
obtain
	$$\begin{array}{ll}
		&\displaystyle w\Big(\Big\{x\notin E:\lambda \mathcal{N}_{1.001\lambda}\Big(\Big\{\sum\limits_{j\in\mathbb{Z}}(H_{j}\phi_{\varepsilon})*b\Big\}_{\varepsilon\in\mathbb{R}^{+}}\Big)^{1/2}(x)>200\alpha\Big\}\Big)\\
		&\quad\leq\displaystyle w\Big(\Big\{x\notin E:\mathcal{S}_{2}\Big(\Big\{\sum\limits_{j\in\mathbb{Z}}(H_{j}\phi_{\varepsilon})*B_{j-s}\Big\}_{\varepsilon\in\mathbb{R}^{+}}\Big)(x)>\alpha\Big\}\Big)\\
		&\quad\quad+\displaystyle w\Big(\Big\{x\notin E:\lambda N_{\lambda/3}\Big(\Big\{\sum\limits_{s\geq 200}\sum\limits_{j\geq k}H_{j}*B_{j-s}\Big\}_{\varepsilon\in\mathbb{R}^{+}}\Big)^{1/2}(x)>12\alpha\Big\}\Big)\\
		&\quad\quad\quad\leq\displaystyle w(\mathcal{G}_{\alpha/2}^{s,1})+w(\mathcal{G}_{\alpha/2}^{s,3})+\sum\limits_{s\geq s_{0}+1}w(\mathcal{E}^{s,1}_{\epsilon_{s}\alpha})
		+\sum\limits_{s\geq s_{0}+1}w(\mathcal{E}^{s,3}_{\epsilon_{s}\alpha})\\
		&\quad\quad\quad\quad\lesssim_d\displaystyle\alpha^{-1}\Big(s_{0}(\delta_{2})+\sum\limits_{s\geq s_{0}(\delta_{3})+1}(\epsilon_{s})^{-2}2^{-\delta_{3}(1-\theta)s}\Big)\|f\|_{L^1(M_{{\frac{1}{\theta}}}w)}\\
		&\lesssim_d[w]_{A_{1}}[w]_{A_{\infty}}\log_{2}([w]_{A_{\infty}}+1)\alpha^{-1}\|f\|_{L^1(w)},
	\end{array}$$
	which gives (5.2) and completes the proof of Theorem \ref{thm1.3}. $\hfill\Box$
	
	\medskip

	\quad\hspace{-20pt}{\bf Acknowledgements.}
	Ting Chen was supported by the National Natural Science Foundation of China
	(Grant No. 12271267) and the Fundamental Research Funds for the Central
	Universities. Feng Liu was supported by the Natural Science Foundation of
	Shandong Province (Grant No. ZR2023MA022) and the National Natural Science
	Foundation of China (Grant No. 12326371).
	
	\medskip

	\quad\hspace{-20pt}{\bf Data Availability}
	Data sharing is not applicable to this article as no datasets were generated
	or analysed during the current study.
	
	\bigskip
	
	\quad\hspace{-20pt}{\bf Declarations}
	
	\bigskip
	\quad\hspace{-20pt}{\bf Conflict of interest} The authors declare no conflict of interest.
	
	\medskip
	\quad\hspace{-20pt}{\bf Ethical Approval} The declaration for ethical approval is not applicable.

\end{document}